\documentclass[a4paper]{amsart}
\usepackage{amssymb,graphicx,pgf,tikz,mathtools}
\usetikzlibrary{positioning,arrows}
\theoremstyle{plain}
\newtheorem{theorem}{Theorem}
\newtheorem{lemma}[theorem]{Lemma}
\newtheorem{proposition}[theorem]{Proposition}

\newtheorem{corollary}[theorem]{Corollary}
\newtheorem{claim}[theorem]{Claim}
\theoremstyle{definition}
\newtheorem{definition}[theorem]{Definition}

\newtheorem{remark}[theorem]{Remark}

\numberwithin{theorem}{section}
\renewcommand{\phi}{\varphi}
\renewcommand{\tilde}{\widetilde}
\newcommand{\htop}{h_{\textrm{\tiny top}}}
\newcommand{\hmu}{h_{\textrm{\tiny Liouville}}}
\def\etc.{\emph{et\thinspace c.}}
\DeclareMathOperator{\Id}{Id}

\DeclareMathOperator{\dist}{\mathrm{d}}

\newcommand{\dd}{\mathop{}\!\mathrm{d}}

\def\R{\mathbb{R}}

\def\Mc{\mathcal{M}}

\let\setminus\smallsetminus
\let\emptyset\varnothing
\newcommand{\Lie}{\mathcal{L}}

\def\dfn{\mathbin{{:}{=}}}
\def\nfd{\mathbin{{=}{:}}}

\let\rest\restriction

\def\geodesic{\mathfrak c}
\usepackage{hyperref}
\usepackage{cleveref}
\usepackage{nicefrac}

\begin{document}
\title{Orbit growth of contact structures after surgery}
\author{Patrick Foulon}
\address{Centre International de Rencontres Math\'ematiques, UMS 822, and Institut de Math\'ematiques de Marseille, UMR 7373, 13453 Marseille, France}
\email{foulon@cirm.univ-mrs.fr}
\author{Boris Hasselblatt}
\address{Department of Mathematics,Tufts University,
Medford, MA 02155,
USA}
\email{Boris.Hasselblatt@tufts.edu}
\thanks{Partially supported by the Committee on Faculty Research Awards of Tufts University}
\author{Anne Vaugon}
\address{Laboratoire de Math\'ematiques d'Orsay, Univ. Paris-Sud, CNRS, Universit\'e Paris-Saclay, 91405 Orsay, France}
\thanks{Partially supported by ANR QUANTACT}
\email{anne.vaugon@math.u-psud.fr}
\begin{abstract}
Investigation of the effects of a contact surgery construction and of invariance of contact homology reveals a rich new field of inquiry at the intersection of dynamical systems and contact geometry. We produce contact 3-flows not topologically orbit-equivalent to any algebraic flow, including examples on many hyperbolic 3-manifolds, and we show how the surgery produces dynamical complexity for any Reeb flow compatible with the resulting contact structure. This includes exponential complexity when neither the surgered flow nor the surgered manifold are hyperbolic. We also demonstrate the use in dynamics of contact homology, a powerful tool in contact geometry.
\end{abstract}
\maketitle
\tableofcontents

\section{Introduction}\label{SIntro}
This paper is a sequel of~\cite{FHLegendrian} in which the authors decribed a surgery construction adapted to contact flows. This construction was originally conceived as a source of uniformly hyperbolic contact flows. However it turns out that the surgered flows exhibit more noteworthy dynamical properties than orginally observed and that interesting consequences of the surgery arise even when the initial or resulting flow are not hyperbolic. Thus the primary interest in this contact surgery may be as a rich source of contact flows exhibiting new phenomena from both dynamical and contact points of view.

The starting point of our surgery is the unit tangent bundle \(M\) of a surface of negative and (mainly, but not always necessarily) constant curvature equipped with its natural contact structures. This surgery is known to contact-symplectic topologists as a Weinstein surgery, and its desciption in~\cite{FHLegendrian} makes it easier to study dynamical properties (as opposed, for instance, to topological properties).

Our purpose is to expand the understanding of the dynamical effects achieved by the contact surgery from~\cite{FHLegendrian}  in 3 main directions: \begin{itemize}
\item we show that the complexity of the resulting flow exceeds that of the flow on which the surgery is performed (Theorems \ref{THMLargerTopEnt}, \ref{THMLargerSharpGrowth}, \ref{THMFHC}, and \ref{th_courbe_simple_polynom}),
\item we show that much of the complexity of the resulting flow is reflected in the cylindrical contact homology and is therefore realized in \emph{any} Reeb flow associated to the contact structure resulting from the surgery (Theorems \ref{THMFHC}, \ref{th_courbe_simple_entropie}, and \ref{th_courbe_simple_polynom}), and
\item we do this beyond the context of hyperbolic flows (in more than one way---Theorems \ref{th_courbe_simple_entropie} and \ref{th_courbe_simple_polynom}).
\end{itemize}
Taken together, this reveals a much richer field of inquiry at the interface between contact geometry and dynamical systems than was apparent when the surgery construction was conceived.

Contact homology and its growth rate are relevant tools to describe dynamical properties of all Reeb flows associated to a given contact structure. Even if it is not always explicit in the statements, they play a crucial role in the proofs of Theorems \ref{THMFHC}, \ref{th_courbe_simple_entropie}, and \ref{th_courbe_simple_polynom}. A goal of this paper is to demonstrate to dynamicists the use of these powerful tools from contact geometry.

In addition to the dynamical point of view, our study is also motivated by contact geometry as we want to investigate connections between growth properties in Reeb dynamics (generally characterized by the growth rate of contact homology) and the geometry of the underlying manifold.
The simplest model of such a connection is Colin and Honda's conjecture~\cite[Conjecture 2.10]{ColinHonda08}, and some surgeries under study give examples supporting it. Colin and Honda speculate that the number of Reeb periodic orbits of universally tight\footnote{see \cref{SDefinitionsNotations}} contact structures on hyperbolic manifolds grows at least exponentially with the period.
More generally, one may look for sources of exponential or polynomial behavior of contact homology. Our starting point, the unit tangent bundle of an hyperbolic surface, is a transitional example as it carries two special contact structures, one with an exponential growth rate for contact homology and one with a polynomial growth rate. We prove (Theorems \ref{th_courbe_simple_entropie} and \ref{th_courbe_simple_polynom}) that some surgeries lead to two coexisting contact forms on the surgered manifold with exponential and polynomial growth rates and therefore give new examples of transitional manifolds with respect to growth rate.
Note that these examples do not include hyperbolic manifolds (and are therefore compatible with Colin and Honda's conjecture).

The principal results in this article were obtained in 2014, and we here integrate it with work by others that was done contemporaneously \cite{Alves1,Alves2,Alves3}.

\subsubsection*{Structure of the paper}
In \cref{SResults} we cover the background material and present our main results. Specifically, \cref{SBSNewContact} presents and elaborates our earlier results \cite{FHLegendrian}, and \cref{SBSClosedOrbits} introduces the resulting complexity increase of the surgered geodesic flow. \cref{SBSClosedOrbitsReeb} finally describes how cylindrical contact homology forces complexity of Reeb flows with the same contact structure, introduces our surgery on the fiber flow, and discusses the relation of our results to other works on contact surgery and Reeb dynamics.

The construction of contact surgery is recalled in \cref{SOutlineSurgery}, which also contains some preliminary results on the dynamics of the surgered flow and the proof of \cref{THMLargerTopEnt}.

In \cref{SECContHom} we define contact homology and its growth rate. This enables us to prove \cref{THMFHC} in \cref{SECOrbitGrowthFHC}, \cref{th_courbe_simple_entropie} in \cref{SECOrbitGrowth_simple_geodesic} and \cref{th_courbe_simple_polynom} in \cref{SMoreContactFlows}.

\subsubsection*{Acknowledgements}
We thank Marcelo Alves, Fr\'ed\'eric Bourgeois, Patrick Massot and Samuel Tapie for useful discussions and helpful advice. Boris Hasselblatt is grateful for the support of the ETH, which was important as we finalized this work.

\section{The results}\label{SResults}

\subsection{Definitions and notations}\label{SDefinitionsNotations}

A manifold is said to be \emph{closed} if it is compact and has no boundary.

A $C^\infty$ 1-form $\alpha$ on a 3-manifold $M$ is called a \emph{contact form} if\/ $\alpha\wedge \dd \alpha$ is a volume form. The associated plane field $\xi\dfn\ker \alpha$ is
a cooriented \emph{contact structure}, and \((M,\xi)\) is called a contact manifold..
The geometric object under study in contact geometry is the contact structure (as opposed to the contact form).
Note that for a given contact structure $\xi=\ker(\alpha)$ the contact forms with kernel\/ $\xi$ are exactly the forms $f\alpha$ where $f\in\mathcal C^\infty(M, \mathbb R\setminus\{0\})$. Additionally, if\/ $\alpha\wedge \dd \alpha$ is a volume form then $f\alpha\wedge \dd(f\alpha)$ is also a volume form for any $f\in\mathcal C^\infty(M, \mathbb R\setminus\{0\})$.
A curve tangent to $\xi$ is said to be \emph{Legendrian}.

The \emph{Reeb vector field} associated to a contact form $\alpha$ is the vector field $R_\alpha$ such that $\iota_{R_\alpha}\alpha=1$ and
$\iota_{R_\alpha}\dd \alpha=0$. Its flow is called the \emph{Reeb flow} (and it preserves \(\alpha\) because \(\Lie_{R_\alpha}\alpha=\iota_{R_\alpha}\dist \alpha=0\)). Note that the Reeb vector field is associated to a contact form $\alpha$: if we consider another contact form $\alpha'=f\alpha$ where $f\in\mathcal C^\infty(M,\mathbb R\setminus\{0\})$, then $\dd\alpha'=\dd f\wedge\alpha+f\dd\alpha$ and the condition $\iota_{R_{\alpha'}}\dd \alpha'=0$ implies that $R_{\alpha}$ and $R_{\alpha'}$ are not collinear unless $f$ is constant. \emph{A} Reeb field on a contact manifold \((M,\xi)\) is the Reeb field of any contact form \(\alpha\) with \(\xi=\ker\alpha\). By Libermann's Theorem \cite{Libermann} on contact Hamiltonians (see for instance \cite[Theorem 2.3.1]{Geiges}), these are exactly the nowhere-vanishing vector fields transverse to $\xi$ whose flows preserve~$\xi$.

A Reeb vector field (or the associated contact form) is said to be \emph{non\-degenerate} if all periodic orbits are non\-degenerate (1 is not an eigenvalue of the differential of the Poincar\'e map). A Reeb vector field (or the associated contact form or the associated contact structure) is said to be \emph{hypertight} if there is no contractible periodic Reeb orbit.
One can always perturb a contact form into a non\-degenerate contact form. Hypertightness is much more restrictive.

Contact structures on $3$-manifolds can be divided into two classes: tight contact structures and overtwisted contact structures. This fundamental distinction is due to Eliashberg~\cite{Eliashberg89} following Bennequin~\cite{Bennequin83}. Tight contact structures are the contact structures that reflect the geometry of the manifolds and this article focuses on them.
A contact structure $\xi$ is said to be \emph{overtwisted} if there exists an embedded disk tangent to $\xi$ on its boundary. Otherwise $\xi$ is said to be \emph{tight}. \emph{Universally tight} contact structures are those with a tight lift to the universal cover. Universally tight and hypertight~\cite{Hofer} contact structures are always tight. All the contact structures considered in this paper are hypertight and therefore tight.

We recall from \cite{FHLegendrian} a contact surgery on a Legendrian curve $\gamma\in S\Sigma$ derived from a closed geodesic \(\geodesic\) on a hyperbolic surface \(\Sigma\). This corresponds to a $(1,-q)$ Dehn-surgery and results in a new manifold $M_S$ with a contact form $\alpha_A$. The construction is presented in \cref{SOutlineSurgery}. The Reeb flow $R_{\alpha_A}$ is Anosov if\/ $q$ is positive---and only then (\cref{PRPNotAnosovIfq<0}).

\begin{definition}[\cite{KatokHasselblatt}]\label{DEFAnosov}
Let $M$ be a manifold and $\varphi\colon\R\times M\to M$ a smooth flow with nowhere vanishing generating vector field \(X\). Then $\varphi$ (and also \(X\)) is said to be an Anosov flow if the tangent bundle $TM$ (necessarily invariantly) splits as $TM=\R X\oplus E^+\oplus E^-$ (the \emph{flow}, \emph{strong-unstable} and \emph{strong-stable directions}, respectively), in such a way that there are constants $C>0$ and $\eta>1>\lambda>0$ for which
\begin{equation}
\label{AnosovCondition}
\bigl\|D\varphi^{-t}\rest{E^+}\bigr\| \le C\eta^{-t}
\quad\text{and}\quad
\bigl\|D\varphi^t\rest{E^-}\bigr\| \le C\lambda^t
\end{equation}
for $t>0$. The \emph{weak-unstable} and \emph{weak-stable} bundles are
$\R X\oplus E^+$ and $\R X\oplus E^-$, respectively. ($E^\pm$ are then tangent to continuous
foliations $W^\pm$ with smooth leaves.)
\end{definition}

An Anosov flow on a 3-manifold is said to be of \emph{algebraic type} if it is finitely covered by the geodesic flow of a surface of constant negative curvature or the suspension of a diffeomorphism of the 2-torus, and it is called a \emph{contact Anosov flow} if it is a Reeb flow, in which case $E^+\oplus E^-$ is the contact structure and $\alpha$ is said to be Anosov as well. Geodesic flows of Riemannian manifolds with negative sectional curvature are Anosov flows. For surfaces of constant negative curvature it is easy to verify the defining property directly, and we do so at the start of \cref{SIMoreContactFlows}.

In this paper, we show that the complexity of the resulting flow exceeds that of the flow on which the surgery is performed. We measure the complexity of the flow of\/ $X$ via its orbit growth, entropy and cohomological pressure. For a contact form $\alpha$, a free homotopy class $\rho$ and $T>0$, we denote by $N_T^\rho(\alpha)$ the number of\/ $R_\alpha$-periodic orbits in $\rho$ with period smaller than $T$ and $N_T(\alpha)$ the number of\/ $R_\alpha$-periodic orbits with period smaller than $T$. The \emph{orbit growth} of\/ $R_\alpha$ (or the associated flow) is the asymptotic behavior of\/ $N_T(\alpha)$, its exponential growth rate is the topological entropy. We summarize the needed notions and facts in \cref{SBSEntropy}. Cohomological pressure drives orbit growth in a given homology class and is defined in \cref{SGrowthHomologyClass}.

\subsection{New contact flows}\label{SBSNewContact}
We begin with a paraphrase of the main result of the surgery construction from \cite{FHLegendrian} in a way that points to the broader perspective of the present work and make a few initial observations that go further.
\begin{theorem}[{\cite[Theorems 1.6, 1.9]{FHLegendrian}}]\label{THMMain}
On the unit tangent bundle \(M\) of a negatively curved surface, there is a
family of smooth Dehn surgeries, including the
Handel--Thurston surgery, that produce new contact flows. The surgered geodesic flow has the following properties:
\begin{enumerate}
\item
It acts on a manifold that is not a unit tangent bundle.
\item\label{itemNotUTB}
If it is Anosov, it is not ortibt equivalent to an algebraic Anosov flow.
\item\label{itemBMMnotVolume}
If it is Anosov, then its topological and volume entropies differ, or, equivalently, the measure of maximal entropy is always singular \cite{Foulon}.
\item
If it is Anosov and if the surgered manifold is hyperbolic, then non empty free homotopy class $\rho$ of closed orbits is infinite and is an isotopy class\rlap,\footnote{Each closed orbit is related to at most finitely many others by the pair being the boundary of an embedded cylinder \cite{BarthelmeFenley}. (This relation is neither transitive nor reflexive.) For comparison, isotopy only ensures that the circles in question are the boundary components of an \emph{immersed} cylinder.} moreover, there exist $a_1, c_1,a_2,c_2>0$ such that
\[\frac{1}{a_2}\ln(T)-c_2\leq N_T^\rho(\alpha_A)\leq a_1\ln(T)+c_1\]
for all\/ $T>0$, where $\alpha_A$ is the contact form defined on the surgered manifold. \cite[Theorem A]{Fenley}, \cite[Remark 5.1.16, Theorem 5.3.3]{BarthelmeDiss}, \cite{BarthelmeFenley},  \cite[Theorem~F]{BarthelmeFenley2}.
\end{enumerate}
\end{theorem}

That these surgeries produce contact flows on hyperbolic manifolds is a corollary of the two following theorems.

\begin{theorem}[Thurston~{\cite[Theorem 5.8.2]{ThurstonPrinceton, ThurstonBAMS}}, Petronio and Porti~{\cite{PetronioPorti}}]
For all but infinitely many slopes, Dehn filling a hyperbolic $3$-manifold gives rise to a hyperbolic manifold.
\end{theorem}

\begin{theorem}[Folklore {\cite[Theorem 1.12]{FHLegendrian}}]\label{THMhypknotcomplement}
Suppose $\Sigma$ is a hyperbolic surface, $\pi\colon S\Sigma\to\Sigma$ its
unit tangent bundle, $\gamma\colon S^1\to M$ continuous such that
$\geodesic\dfn\pi\circ\gamma$ is a closed geodesic that is not the same geodesic
traversed more than once and such that $\ell\cap \geodesic\neq\emptyset$ whenever
$\ell$ is a noncontractible closed curve. Then
$S\Sigma\setminus(\gamma(S^1))$ is a hyperbolic manifold.
\end{theorem}
Nonetheless, there exist infinitely many closed orientable hyperbolic manifolds of dimension $3$ which do not support an Anosov flow \cite[Theorem A]{RobertsShareshianStein}.
Additionally, since there are only finitely many homotopy classes of tight contact structures on a 3-manifold \cite[Th\'eor\`eme 1]{ColinGirouxHonda} and the contact structures with an Anosov Reeb flow are tight as they are hypertight (\cite{PlanteThurston}, \cite[p.~18]{BarbotHDR}), there exist only finitely many homotopy classes of contact Anosov flows on a given $3$-manifold. On hyperbolic $3$-manifolds the same goes for isotopy classes \cite[Th\'eor\`eme 2]{ColinGirouxHonda}. We do not know if the surgery from \cref{THMMain} can produce different contact structures on the same manifold.

\begin{remark}
The dynamical properties of the flow after surgery differ from the properties of Anosov algebraic flows. Indeed, for algebraic flows, free homotopy classes of closed orbits are finite. For geodesic flows no two (parametrized) orbits are homotopic, though rotating the tangent vector through $\pi$ isotopes each to its flip, which has the same \emph{image} as another orbit (the same geodesic run backwards), and only in suspensions are all free homotopy classes of \emph{images} of orbits singletons \cite[Corollary 4.3]{BarthelmeFenley2}.
\end{remark}

Our surgery corresponds to a $(1,-q)$-Dehn surgery and produces Anosov Reeb flows for $q>0$. As part of our study focuses on the $q<0$-case, it is important to note the following.

\begin{proposition}
Some surgeries from \cref{THMMain} produce flows that are not Anosov (\cref{PRPNotAnosovIfq<0}).
\end{proposition}

In the case $q=1$, this surgery is the standard Weinstein surgery as defined by Weinstein~\cite{Weinstein91} in 1991 simplifying Eliashberg's work~\cite{Eliashberg90} of 1990 (see \cite[Chapter 6]{Geiges} for more details). The surgery $(1,q)$ for any $q$ can be deduced from this construction. A direct construction for any $q$ using Giroux theory of convex surfaces can be found in \cite{DingGeiges}.

In answer to a question of Serge Troubetzkoy we here note:
\begin{proposition}
There are analytic Anosov flows as described in \cref{THMMain}.
\end{proposition}
\begin{proof}
The contact form is smooth and can hence be approximated by analytic ones.
The contact property of the form and the Anosov property of its Reeb flow are open.
\end{proof}

\begin{remark}
Another perspective on the connection with the Handel--Thurston construction is that our result implies in particular that the Handel--Thurston examples are topologically orbit-equivalent to contact flows.
\end{remark}

\begin{remark}\label{REMBernoulli}
For context we recall here that contact Anosov flows have the \emph{Bernoulli} property \cite{KatokBurns,OrnsteinWeiss,ChernovHaskell} and exponential decay of correlations \cite{Liverani}. The Bernoulli property and the Ornstein Isomorphism Theorem \cite{Ornstein} imply that the flows we obtain from our surgery are measure-theoretically isomorphic to the original contact Anosov flow up to a constant rescaling of time, the constant being the ratio of the Liouville entropies. (This answers a question of Vershik.)
\end{remark}

\subsection{Production of closed orbits for contact Anosov flows}\label{SBSClosedOrbits}

\subsubsection{Impact on entropy}\label{SBSEntropy}
We continue with new results about the features of the contact Anosov flows from \cite{FHLegendrian} to the effect that the surgery of \cref{THMMain} produces ``exponentially many'' closed orbits. We preface these statements by a brief summary of the needed notions and facts pertinent to entropy.
\begin{itemize}
\item The \emph{topological entropy} of an Anosov flow (or of the vector field that generates it) equals the exponential growth rate of the number of periodic orbits; in our case this means that \(\htop(R_\alpha)=\lim_{T\to\infty}\frac1T\log N_T(\alpha)\).
\item The entropy \(h_\mu(\varphi^t)\) of a flow \(\varphi^t\) with respect to an invariant Borel probability measure \(\mu\) (also referred to as the entropy of \(\mu\) with respect to \(\varphi^t\)) does not exceed the topological entropy of \(\varphi^t\).\footnote{Indeed, the topological entropy is the supremum of the entropies of invariant Borel probability measures (Variational Principle).}
\item If a flow-invariant Borel probability measure \(\mu\) is absolutely continuous with respect to a smooth volume, then we say it is a Liouville measure and write \(\hmu\dfn h_\mu\).
\item For the geodesic flow \(g^t\) of a surface we have \(\hmu(g^t)=\htop(g^t)\) if (and only if \cite{Foulon,KatokEntClGeod,KatokConformal}) the curvature is constant.
\item Scaling of time: if \(s\in(0,\infty)\), then \(\hmu(sX)=s\hmu(X)\) and  \(\htop(sX)=s\htop(X)\).
\item More generally, there is Abramov's formula: the entropy of a time change \(gX\) of a nonzero vector field \(X\) with respect to a  \(gX\)-invariant probability measure \(\mu_g\)  canonically associated with an \(X\)-invariant Borel probability measure \(\mu\) is
\begin{equation}\label{eqAbramov}
h_{\mu_g}(gX)=h_\mu(X)\int g\dd\mu.
\end{equation}
This means that comparisons of the intrinsic dynamical complexity of these vector fields are meaningful only when \(\int g=1\).
\item Pesin entropy formula \cite{BarreiraPesin}: For a volume-preserving flow \(\varphi^t\) with 1-di\-men\-sio\-nal expanding direction, \(\hmu(\varphi^t)\) equals the positive Lyapunov exponent of the flow \cite{BarreiraPesin}, \cite[Definition S.2.5]{KatokHasselblatt}, which is (a.e.) defined as the exponential growth rate of unstable vectors under the flow
and as a function of time.
\end{itemize}
\begin{theorem}\label{THMLargerTopEnt}
If\/ $\psi^t$ is a contact Anosov flow obtained from the geodesic flow $g^t$ of a compact oriented surface of constant negative curvature by the surgery in \cref{THMMain} (generated by the vector field in \eqref{eqDEFXh}), then its topological entropy is strictly larger. Indeed, $\htop(\psi^t)>\hmu(\psi^t)\ge\hmu(g^t)=\htop(g^t)$.

Since $\htop$ measures the exponential growth rate of periodic orbits of a hyperbolic dynamical system, the number $N_T(\psi^t)$ of\/ $\psi^t$-periodic orbits of period $t\le T$ (of up to a given length) grows at a larger exponential rate than $N_T(g^t)$.
\end{theorem}
\begin{remark}\label{REMBishopHughesVinhageYang}
The strict inequality in \cref{THMLargerTopEnt} is obtained by contraposition of a rigidity result \cite{Foulon}, so we do not know by how much the topological entropy increases through our surgery. Recently, Bishop, Hughes, Vinhage and Yang suggested to provide effective lower bounds for this entropy-increase by using cutting sequences in the spirit of Series.
\end{remark}

\subsubsection{Growth in homology classes}\label{SGrowthHomologyClass}
In a self-contained digression, we can give rather more detailed information about orbit growth in homology classes.

\begin{theorem}\label{THMLargerSharpGrowth}
If\/ $\psi^t$ is a contact Anosov flow obtained from the geodesic flow $g^t$ of a compact oriented surface of constant negative curvature by the surgery in \cref{THMMain} (generated by the vector field in \eqref{eqDEFXh}), then
\[N_T^\zeta(\psi^t)\big/N_T^\eta(g^t)\xrightarrow[T\to\infty]{\text{\tiny{exponentially}}}\infty\]
for any homology classes $\zeta$ for $\psi^t$ and $\eta$ for $g^t$ (where $N_T^\zeta(\psi^t)$ and $N_T^\eta(g^t)$ count the number of periodic orbits orbit with period $\leq T$ in the \emph{homology} classes $\zeta$ and $\eta$).
\end{theorem}

The proof derives from the notion of \emph{cohomological pressure}.

\begin{definition}[{\cite[Theorem 1(iii), p.\ 398]{Sharp}}]
The \emph{cohomological pressure} of\/ $\varphi^t$ is \looseness-1
\[
P(\varphi^t)\dfn\inf_{[b]\in H^1(M,\R)}\Big\{\sup_{\mu\in\Mc(\varphi^t)}\big\{h_\mu(\varphi^t)+\int b(X)\dd\mu\big\}\Big\},
\]
where
$\Mc(\varphi^t)$ is the set of\/ $\varphi^t$-invariant Borel probability measures.
\end{definition}

\begin{remark}
The cohomological pressure is the usual pressure of the function \(b(X)\) and the abose formula is well-defined for a cohomology class $[b]$. Indeed, here, \(H^1(M,\R)\) is the first de Rham cohomology group, and the integral is the Schwartzman winding cycle, which is well-defined for a closed 1-form when $\mu$ is $\varphi^t$-invariant; the supremum is unaffected by addition of an exact form to~\(b\).
\end{remark}

Contact Anosov flows satisfy
\begin{equation}
\htop(\varphi^t)\ge P(\varphi^t)\ge\hmu(\varphi^t)\quad\text{\cite[Corollary 1]{FangTherm}}.\label{eqFanginequality}
\end{equation}
\begin{theorem}[{\cite[Theorem 5.3]{FangTherm}}]\label{THMFangCohomPressure}
If\/ $\varphi^t$ is a flow such as the ones obtained in \cref{THMMain}, then
$P(\varphi^t)>\hmu(\varphi^t)$.
\end{theorem}
Thus, the conclusion of \cref{THMLargerTopEnt} is strengthened to
\[\htop(\varphi^t)\ge P(\varphi^t)>\hmu(\varphi^t)\ge\hmu(g^t)=\htop(g^t).\]
This makes it possible to amplify the observation about increased orbit growth and prove \cref{THMLargerSharpGrowth}.
Indeed, contact Anosov flows are homologically full\footnote{I.e., every homology class contains a closed orbit} \cite[Proposition
1]{FangTherm}, and, for homologically full flows, cohomological pressure drives orbit growth in a given homology class $\zeta$ \cite[Theorem 1]{Sharp}:
\begin{equation}\label{eqSharpCounting}
N_T^\zeta(\phi_t)\sim C(\zeta)\frac{e^{TP(\varphi_t)}}{T^{1+\frac{b_1}{2}}}\text{ as }T\to\infty,
\end{equation}
where $b_1$ is the first Betti number of the underlying manifold.

\subsection{Production of closed orbits for any Reeb flow}\label{SBSClosedOrbitsReeb}

We now broaden the scope far beyond hyperbolic dynamics by beginning to involve contact geometry in a serious fashion. Specifically, the existence of well-understood Reeb flows, such as those in \cref{THMMain}, allows us to control all the other Reeb flows associated to the same contact structure in terms of entropy or orbit growth. We transcend hyperbolicity because we describe here our results concerning dynamical properties of Reeb flows associated to all (or a subclass of) contact forms after a contact surgery. These flows need not be hyperbolic even if the contact structure arises from an Anosov flow.

\subsubsection{Orbit growth from Anosov Reeb flows}

This section presents an archetype of theorem deriving properties for all Reeb flows from stronger properties for one Reeb flow. Our results described in \cref{SOrbitGrowthReebFlows} can be seen as extension of this theorem. It can be applied to some of the contact flows described in \cref{THMMain}.

The existence of Anosov Reeb flows is a source of exponential orbit growth for all Reeb flows as proved by Alves or Macarini and Paternain~\cite[Theorem 2.12.]{MacariniPaternain}.

\begin{theorem}[{Alves \cite[Corollary~1]{Alves3}}]\label{THMAlvesAnosov}
If one Reeb flow for a compact contact 3-manifold \((M,\xi)\) is Anosov, then every Reeb flow on \((M,\xi)\) has positive topological entropy.
\end{theorem}

\begin{remark}
Alves also obtains lower bounds for the entropy: for $\alpha=f\alpha_0$ with $f>0$, we get $h(R_\alpha)\geq a/\max(f)$ where $a$ is some growth rate associated to $R_{\alpha_0}$. Note also that these estimates can not be obtained by the Abramov formula, which determines the measure-theoretic entropy of a time-change because different Reeb fields for a contact structure need not be collinear.
\end{remark}

The standard contact structure on the unit tangent bundle of a hyperbolic surface has an Anosov Reeb flow and therefore, by \cref{THMAlvesAnosov}, all its other Reeb flows have positive entropy and their orbit growth is at least exponential. In particular, \cref{THMAlvesAnosov} applies to the contact structures obtained in \cref{THMMain} on hyperbolic manifolds: these are examples satisfying the Colin--Honda conjecture, and on non\-hyperbolic manifolds, for instance, when the surgery is associated to a simple geodesic. We give a slightly different proof of this result in \cref{SECContHom}.

\subsubsection{Orbit growth from contact homology}\label{SOrbitGrowthReebFlows}

We now present our results and extend \cref{THMAlvesAnosov} in two different settings
\begin{enumerate}
\item when the Reeb flow after surgery is Anosov, we study orbit growth in free homotopy classes;
\item when the geodesic associated to the surgery is a simple curve, we prove positivity of entropy for any contact form (and any surgery).
\end{enumerate}

Let us describe our results in the first setting. The following result can be seen as a corollary
of the invariance of contact homology and the Barthelm\'e--Fenley estimates from \cite[Theorem~F]{BarthelmeFenley2} in the non\-degenerate case, and of Alves' proof of Theorem~1 in \cite{Alves3} and the Barthelm\'e--Fenley estimates from \cite[Theorem~F]{BarthelmeFenley2} in the degenerate case.

\begin{theorem}\label{THMFHC}
Let $(M_S,\xi_S=\ker(\alpha_A))$ be a contact manifold obtained after a non-trivial contact surgery such that $\alpha_A$ is Anosov. Let $\rho$ be a primitive free homotopy class containing at least one $R_{\alpha_A}$-periodic orbit. Then for all contact forms $\alpha$ on $(M_S,\xi_S)$, $\rho$ contains infinitely many $R_\alpha$-periodic orbits.
Additionally,
\begin{enumerate}
\item if\/ $\alpha$ is non\-degenerate, there exist $a>0$ and $b\in\mathbb R$ such that
$N_T^\rho(\lambda)\geq a\ln(T)+b$ for all\/ $T>0$,
\item if\/ $\alpha$ is degenerate and $M_S$ is hyperbolic, there exist $a>0$ and $b\in\mathbb R$ such that
$N_T^\rho(\lambda)\geq a\ln(\ln(T))+b$
for all\/ $T>0$.
\end{enumerate}
\end{theorem}

\begin{remark}
In fact, for $\alpha$ non\-degenerate, we will prove $N_T^\rho(\alpha)\geq N_{CT}^\rho(\alpha_A)$ for some $C>0$ and for all\/ $T>0$ and use the Barthelm\'e--Fenley result. Therefore better control of\/ $N_{T}^\rho(\alpha_A)$ in some free homotopy classes will lead to better estimates.
\end{remark}

\begin{remark}
There is no hope to obtain a upper bound on $N_{T}^\rho(\alpha_A)$ for all contact forms as the number of Reeb periodic orbits can always be increased by creating many periodic orbits in a neighborhood of a preexisting periodic orbit.
\end{remark}

\begin{remark}\label{RMKGR}
If the manifold is not hyperbolic, the Barthelm\'e and Fenley estimates are weaker as the upper bound is linear. The proof of \cref{THMFHC} can be adapted to this situation but leads to weak control of the growth of periodic orbits in a given homotopy class for degenerate contact forms.
\end{remark}

We now turn to our second setting and assume the geodesic associated to the surgery is a simple curve. Note that we do not assume that the Reeb flow is Anosov and therefore consider any $(1,q)$-Dehn surgery. Additionally, note that $M_S$ is never a hyperbolic manifold in this setting. Our main theorem is the following.

\begin{theorem}\label{th_courbe_simple_entropie}
If\/ $(M_S,\alpha_A)$ is a contact manifold obtained from contact surgery along a simple geodesic, then any Reeb flow of\/ \((M_S,\ker(\alpha_A))\) has positive topological entropy.
\end{theorem}

In particular, the number of periodic orbits grows at least exponentially with respect to the period. The proof of this theorem is based on Alves' work~\cite{Alves1}. In the same paper, Alves obtains the same result when the associated geodesic is separating~\cite[Section~4 and Theorem~2]{Alves1}. Our strategy of proof is similar to that of Alves.

Floer type homology and especially contact homology are the main tools to control Reeb periodic orbits of all contact forms associated to a contact structure. The contact homology of a ``nice'' contact form $\alpha_0$ is the homology of a complex generated by $R_{\alpha_0}$-periodic orbits and therefore encode dynamical properties of the Reeb vector field (contact homology is described in~\cref{SECContHom}).

The growth rate of contact homology makes it possible define the polynomial behavior of a contact structure. We now focus on examples obtained by surgery exhibiting polynomial growth.

\subsubsection{Coexistence of diverse contact flows}\label{SIMoreContactFlows}

We first introduce the three Reeb flows that naturally appear on the unit tangent bundle of a constantly curved surface of higher genus. This is elementary but not commonly presented.
On the unit tangent bundle of a hyperbolic surface, there is a canonical framing consisting of \(X\), the vector field on \(S\Sigma\) that generates the geodesic flow, of \(V\), the vertical vector field (pointing in the fiber direction), and of \(H\dfn[V,X]\). It satisfies the classical \emph{structure equations}
\begin{equation}\label{eqStructureEquations}
[V,X]=H,\quad[H,X]=V,\quad[H,V]=X.
\end{equation}
One can check these by using that in the \(\operatorname{PSL}(2,\R)\)-representation of \(S\tilde\Sigma\), these vector fields are given by
\[
X\sim\begin{pmatrix}\nicefrac12&0\\0&-\nicefrac12\end{pmatrix},\quad H\sim\begin{pmatrix}0&\nicefrac12\\\nicefrac12&0\end{pmatrix},\quad V\sim\begin{pmatrix}0&-\nicefrac12\\\nicefrac12&0\end{pmatrix}.
\]
The structure equations imply that $e^\pm\dfn V\pm
H$ satisfies $[X,V\pm H]=\mp e^\pm$, so if a vector field $f\cdot e^\pm$ along an orbit of \(X\) is
invariant under the geodesic flow, then $0=[X,fe^\pm]=(\dot f\mp
f)e^\pm$, where \(\dot f\) is the derivative along the orbit. This means that $ \dot f=\pm f$, so $f(t)=\operatorname{const}e^{\pm t}$. Thus, the differential of
the geodesic flow expands and contracts,
respectively, the directions $e^\pm$; this is the Anosov property and $E^\pm$ is spanned by the vector $e^\pm=V\pm H$.

Of course, in the \(\operatorname{PSL}(2,\R)\)-representation of \(S\tilde\Sigma\), these 3 flows are given by
\begin{align*}
X\leadsto\exp\Big(\begin{pmatrix}\nicefrac12&0\\0&-\nicefrac12\end{pmatrix}t\Big)&=\begin{pmatrix}e^{\nicefrac t2}&0\\0&e^{-\nicefrac t2}\end{pmatrix},\\
H\leadsto\exp\Big(\begin{pmatrix}0&\nicefrac12\\\nicefrac12&0\end{pmatrix}t\Big)&=\begin{pmatrix}\cosh\nicefrac t2&\sinh\nicefrac t2\\\sinh\nicefrac t2&\cosh\nicefrac t2\end{pmatrix},\\
V\leadsto\exp\Big(\begin{pmatrix}0&-\nicefrac12\\\nicefrac12&0\end{pmatrix}t\Big)&=\begin{pmatrix}\cos \nicefrac t2&-\sin \nicefrac t2\\\sin \nicefrac t2&\cos \nicefrac t2\end{pmatrix}
\end{align*}
To see in these terms that $X$ generates a contact flow, define a 1-form $\alpha_0$ by \(\alpha_0(X)=1\) and \(\alpha_0(V)=0=\alpha_0(H)\). For \(Z\in\{V,H\}\) we have
\[
\dist \alpha_0(X,Z)=\underbracket{\Lie_X\overbracket{\alpha_0(Z)}^{\equiv0}}_{=0}-\underbracket{\Lie_Z\overbracket{\alpha_0(X)}^{\equiv1}}_{=0}-\underbracket{\alpha_0(\overbracket{[X,Z]}^{\mathclap{\in-\{V,H\}}})}_{=0}=0,
\]
so $\iota_X\dist \alpha_0\equiv0$.
Additionally \(\alpha_0\wedge\dist \alpha_0(X,V,H)=\alpha_0(X)\dist \alpha_0(V,H)=1\) because
\[
\dist \alpha_0(V,H)=\underbracket{\Lie_V\overbracket{\alpha_0(H)}^{\equiv0}}_{=0}-\underbracket{\Lie_H\overbracket{\alpha_0(V)}^{\equiv0}}_{=0}-\underbracket{\alpha_0(\overbracket{[V,H]}^{=-X})}_{=-1}=1.
\]
Thus, $\alpha_0\wedge\dist \alpha_0$ is a volume form; in fact a volume particularly well adapted to this canonical framing, and $\alpha_0$ is a contact form. Additionally, \(X=R_{\alpha_0}\).

Likewise, one can check that the $1$-forms $\beta$ and $\gamma$ defined by \(\beta(V)=1\) and \(\beta(X)=0=\beta(H)\), and \(\gamma(H)=1\) and \(\gamma(X)=0=\gamma(V)\) are also contact forms. Their Reeb vector fields are $R_\beta=V$ and $R_\gamma=H$. Note that $\gamma = \dd \alpha_0(V,\cdot)$ and $\beta = -\dd \alpha_0(H,\cdot)$. Additionally, the orientation given by $\beta\wedge\dd\beta$ is the opposite of the orientation given by $\alpha_0\wedge\dist \alpha_0$; therefore $\alpha_0$ and $\beta$ define different contact structures. By contrast, $\alpha_0$ and $\gamma$ define isotopic contact structures. Indeed, let $\psi^t$ be the flow of\/ $V$. Then,
\begin{align*}
(\psi_t)_*X&=\cos\nicefrac t2X+\sin\nicefrac t2H
\quad\text{and}\\
(\psi_t)_*H&=\cos\nicefrac t2H-\sin\nicefrac t2X,
\end{align*}
thus
\[(\psi_t)_*\alpha_0=\cos\nicefrac t2\alpha_0+\sin\nicefrac t2\gamma \]
as the two contact forms coincide on $(\psi_t)_*X$, $(\psi_t)_*H$ and $(\psi_t)_*V=V$.
So it suffices to study the geodesic flow as the leading representative of this \(S^1\)-family of contact Anosov flows. Geometrically, this family of flows can be described as: rotate a vector by an angle, carry it along the geodesic it now defines, and rotate back by the same angle. In other words, it is parallel transport for a fixed angle.

Dynamically $R_{\alpha_0}$ and $R_\beta$ are polar opposites: the geodesic flow is hyperbolic and the fiber flow is periodic. The surgery increases the complexity of both, whether or not the twist goes in the correct direction to produce hyperbolicity from the geodesic flow. For the geodesic flow this is \cref{th_courbe_simple_entropie}, and for the fiber flow it is:

\begin{theorem}\label{th_courbe_simple_polynom}
Let \((M_S,\ker(\beta_S))\) be a contact manifold obtained from the contact form for the fiber flow after a non-trivial contact surgery along a simple geodesic. Then the growth rate of contact homology for \((M_S,\ker(\beta_S))\) is quadratic. In particular, any nondegenerate Reeb flow of \((M_S,\ker(\beta_S))\) has at least quadratic orbit growth\footnote{This means that for any nondegenerate contact form $\beta$ such that $\ker(\beta)=\ker(\beta_S)$ the number $N_T(\beta)$ of $R_\beta$-periodic orbits with period smaller than $T$ satisfies $N_T(\beta)\geq a T^2$ for some positive real number $a$.}.
\end{theorem}

\subsubsection{Relation to other works on contact surgery and Reeb dynamics}

Weinstein surgery/handle attachment is an elementary building block and fundamental operation in contact/symplectic topology and has been largely studied from the topological point of view (for instance it can be used to construct specific or tight or fillable contact manifolds).
We only mention here works focusing on the Reeb dynamics.

A description of contact surgery with control of the Reeb vector field can be found in \cite{EtnyreGhrist}, where Etnyre and Ghrist construct tight contact structures and prove tightness using dynamical properties of the Reeb vector field (their desciption is  different from ours as they consider a surgery on a transverse knot and focus on the description of this surgery via tori).

In \cite{BEE}, Bourgeois, Ekholm and Eliashberg describe the effect of a Weinstein surgery on Reeb dynamics and contact homology. More precisely, they prove the existence of an exact triangle in any dimension connecting contact homologies of the initial manifold and the surgered manifold and a third term associated to the attaching sphere and called Legendrian contact homology. However, explicit computations are delicate even for our explicit examples, for instance as the Lengendrian contact homology is the homology of a huge complex. In contrast, our results give precise estimates in Reeb dynamics but for specific examples.

Our work is largely inspired by Alves' work on Reeb dynamics as explained above, note that he himself applied his methods to contact sugery. The study of Reeb flows with positive entropy comes from Macarini and Schlenk~\cite{MacariniSchlenk} of the unit cotangent bundle equipped with the standard contact structure. This has been developed by Macarini and Paternain~\cite{MacariniPaternain}, Alves~\cite{Alves1, Alves2, Alves3} and others. In \cite{AlvesEtAl}, Alves, Colin and Honda relate topological entropy of Reeb flows to the monodromy of an associated open book decomposition.

\section{Surgery and production of closed orbits}\label{SOutlineSurgery}
The surgery in \cite{FHLegendrian} on which this work is based came with some infelicitous conventions and an immaterial sign error, so we recapitulate some of the steps here with more explicit details. This is necessary also as a base for the proof of \cref{THMLargerTopEnt}, and for a supplementary result (\cref{PROPFoliationsOrientability}) that is needed later. Our surgery can be performed in a neighborhood of any Legendrian knot in a contact $3$-manifold.
We start with a description of the surgery in adapted coordinates near a Legendrian and then explain how to obtain such coordinates in the unit tangent bundle of a hyperbolic surface and how they are linked to the stable and unstable bundles.

\subsection{The surgery from the contact viewpoint}\label{SScontact}

Let $(M,\xi=\ker(\alpha))$ be a contact $3$-manifold and let $\gamma$ be a Legendrian knot in $M$. Then there exist coordinates
\[
(t,s,w)\in\Omega\dfn(-\eta,\eta)\times S^1\times(-\epsilon,+\epsilon),
\]
with $0<\epsilon<\eta/2\pi$ on a neighborhood of \(\gamma\) in which $\alpha=\dd t+w\,\dd s$ and $\gamma=\{0\}\times S^1\times\{0\}$. The \emph{surgery annulus} is $\{0\}\times S^1\times (-\epsilon,+\epsilon)$.
Note that in these coordinates $\alpha\wedge\dd\alpha = \d
t\wedge \dd w\wedge \dd s$ and $R_\alpha=\frac{\partial}{\partial t}$, so $\Omega$ is a flow-box chart. The surgeries split this chart into 2 one-sided flow-box neighborhoods of the surgery annulus, and while the initial transition map between these on $\{0\}\times S^1\times(-\epsilon,+\epsilon)$ is the identity, the surgered manifold $M_S$ is defined by imposing the desired twist (or shear) as the transition map on this annulus:
\begin{equation}\label{eqdefF}
F\colon S^1\times(-\epsilon,\epsilon)\to S^1\times(-\epsilon,\epsilon),\quad(s,w)\mapsto(s+f(w),w)
\end{equation}
with $ f\colon[-\epsilon,\epsilon]\to S^1$, $w\mapsto\exp(iqg(w/\epsilon))
$, $q\in\mathbb Z$, $g\colon\R\to[0,2\pi]$ nondecreasing smooth, $0\le g'\le4$ even, and
$g((-\infty,-1])=\{0\}$, $g([1,\infty))=\{2\pi\}$.
We specify that the transition map from $\{t<0\}$ to $\{t>0\}$ is used to identify points $(0^-,x)$ with $(0^+,F(x))$. With this choice one see that $F^*\alpha=\alpha+wf'(w)\,dw$
and hence that
\[F^*\dd \alpha=\dd \alpha\quad\text{and}\quad F^*(\alpha\wedge \dd \alpha)=\alpha\wedge \dd \alpha,\]
so $\alpha\wedge \dd \alpha$ is a well-defined volume on $M_S$. The vector field $R_\alpha$ on $M$ induces the \emph{Handel--Thurston} vector field $X_{\mathit{HT}}$ on $M_S$. Its flow preserves the Liouville volume defined by $\alpha\wedge \dd \alpha$ \cite[Corollary 3.3]{FHLegendrian}, and the total volume of the manifold is not changed by the surgery.

However, we have not yet produced a \emph{contact} flow: $F^*\alpha=\alpha +wf'(w)\dd w$, so $\alpha$ does not induce a contact form on $M_S$. A deformation yields a well-defined contact form $\alpha^\mp_h=\alpha\mp\dd h$ for $\pm t\geq 0$, where
\[
h(t,w)\dfn\frac12{\underbracket{\lambda(t)}_{\mathclap{\quad\lambda\colon\R\to[0,1]\text{ is a smooth bump function\strut}}}}\int_{-\epsilon}^wxf'(x)\,\dd x\text{ on }(-\eta,\eta)\times(-\epsilon, \epsilon)\text{ and }h=0\text{ outside}.
\]
satisfies $dh=\frac12wf'(w)\dd w$ on the surgery annulus and $h\equiv 0$ for $t$ close to $\pm\eta$. Hence $F^*(\alpha_h^+)=\alpha_h^-$ and $\alpha_h^\pm$ induces a contact form $\alpha_A$ on $M_S$. Its Reeb field is a time-change
\begin{equation}\label{eqXh}
R_{\alpha_A}\dfn\dfrac{X_{\mathit{HT}}}{1\pm \dd h(X_{\mathit{HT}})}
\end{equation}
of\/ $X_{\mathit{HT}}$ \cite[Theorem 4.2]{FHLegendrian}, which is well-defined because$\vert\dd h(X_{\mathit{HT}})\vert<1$ if\/ $0<\epsilon<\eta/2\pi$ \cite[Theorem 4.1]{FHLegendrian}. If one considers smaller $\epsilon$, it is possible to impose the condition $\vert\dd h(X_{\mathit{HT}})\vert<1/2$ and we will do so in \cref{SMoreContactFlows}.

The time-change that defines \(R_{\alpha_A}\) is a slow-down near the surgery annulus, which confounds comparisons of dynamical complexity because of the extra factor in Abramov's formula \eqref{eqAbramov}, so we study the vector field
\begin{equation}\label{eqDEFXh}
X_h\dfn cR_{\alpha_A}=R_{\alpha_A/c},
\end{equation}
where \(c\in\R\) is such that \(\displaystyle\int\dfrac{c}{1\pm \dd h(X_{\mathit{HT}})}\alpha\wedge\dd\alpha=1\) to compare entropies.

\subsection{Surgery on unit tangent bundle and Anosov flows}\label{SSAnosov}

We now explain how to perform a contact surgery on the unit tangent bundle of a hyperbolic surface $\Sigma$. Select a closed geodesic $\geodesic\colon S^1\to\Sigma$, $s\mapsto\geodesic(s)$ and consider the Legendrian knot $\gamma$  obtained by rotating the unit vector field along $\geodesic$ by the angle $\theta=\pi/2$. This knot is Legendrain as $H$ is tangent to $\gamma$ (see \cref{fig1}).
\begin{figure}[h]
\begin{center}
\reflectbox{\includegraphics[width=.6\textwidth]{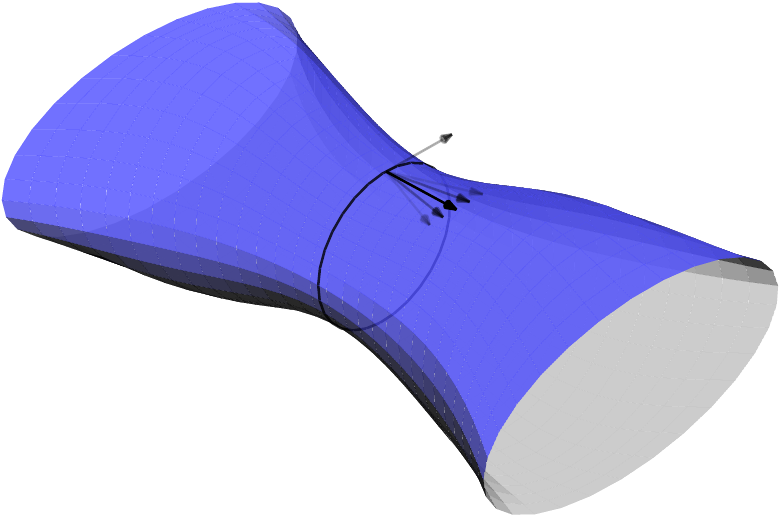}}
\caption{Surgery annulus in the base}
\label{fig1}
\end{center}
\end{figure}
Standard coordinates for $\alpha_A$ near $\gamma$ are obtained by flowing along the vertical field $V$ and then along the geodesic vector field $X$ \cite[Lemma 5.1]{FHLegendrian}: the surgery annulus is contained in the torus $\mathbb T$ above $\geodesic$ (see \cref{FIGSurgeryAnnulus}); it consists of vectors that are almost orthogonal to a chosen geodesic in a surface. Along $\gamma$, $E^+$ is spanned by a vector $V+H$ in the first quadrant.

\begin{figure}[h]
\begin{center}
\footnotesize
\includegraphics[width=.5\textwidth]{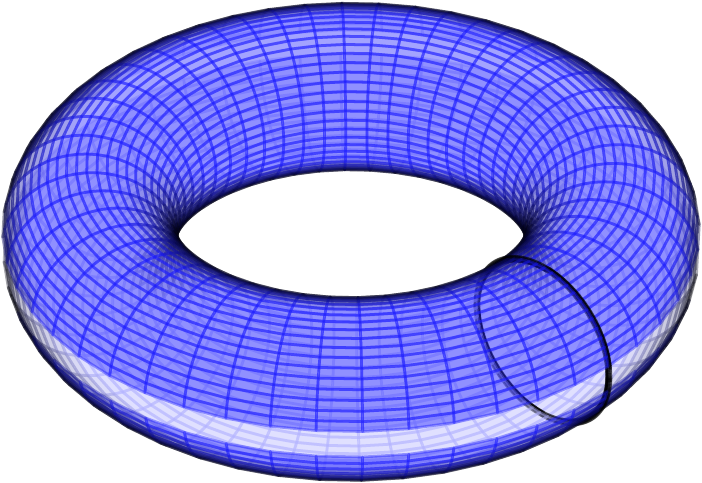}%
\includegraphics[width=.5\textwidth]{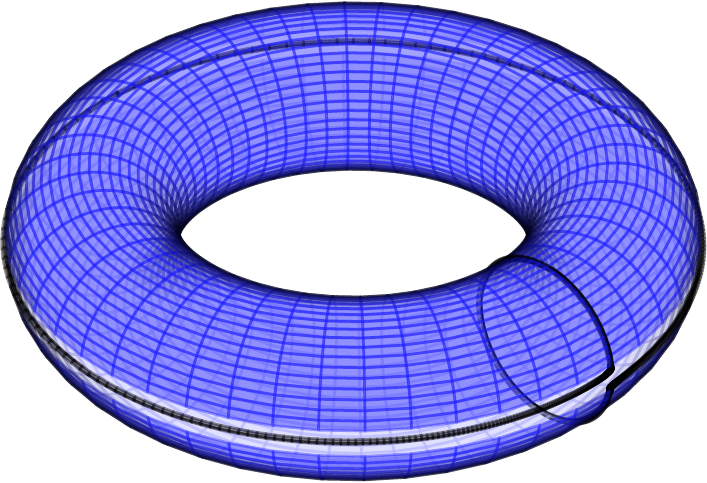}
\normalsize
\caption{Surgery annulus before and after surgery (\(q=1\))}
\label{FIGSurgeryAnnulus}
\end{center}
\end{figure}

To prove that the surgered flow is Anosov, \cite{FHLegendrian} uses Lyapunov--Lorentz metrics \cite[Claim 4.5 and Appendix A]{FHLegendrian}.

\begin{definition} The continuous Lorentz metrics $Q^+$ and $Q^-$ on $M$ are a pair of Lyapunov--Lorentz metrics for the flow $\phi^t$ generated by $X$ if there exists constants $a,b,c,T>0$ such that
\begin{enumerate}
\item $C^+\cap C^-=\emptyset$ where $C^\pm$ is the $Q^\pm$-positive cone;
\item $Q^\pm(X)=-c$;
\item for any $x\in M$, $v\in C^\pm(x)$ and $t>T$, $Q^\pm(D_x\phi^{\pm t}(v))\geq ae^{bt}Q^\pm(v)$
\item for any $x\in M$ $D_x\phi^{\pm T}\left(\overline{C^\pm(x)}\right)\setminus \{0\}\subset C^\pm\left(\phi^{\pm T}(x)\right)$
\end{enumerate}
\end{definition}

\begin{proposition}{\cite[Claim 4.5 and Appendix A]{FHLegendrian}}
A smooth flow $\phi^t$ is Anosov if and only if it admits a pair of Lyapunov--Lorentz metrics $Q^-$ and $Q^+$. The unstable foliation of the flow is then contained in the positive cone $Q^+$ and the stable foliation in the positive cone of\/ $Q^-$
\end{proposition}

For the geodesic flow, one can choose $Q^\pm=\pm\dd w\dd s-c\dd t^2$ in the coordinates $(t,s,w)$. Understanding how the surgery affects the positive cones of\/ $Q^\pm$ is crucial to understand why the condition $q$ positive is essential to obtain an Anosov flow after surgery. We restrict attention to the trace of these cones in the $sw$-plane and consider the geometry of the action of \(F\) by differentiating \eqref{eqdefF} to see the twist (shear) in \((s,w)\)-coordinates:
\[
DF=\begin{pmatrix}1&f'(w)\\0&1\end{pmatrix}.
\]
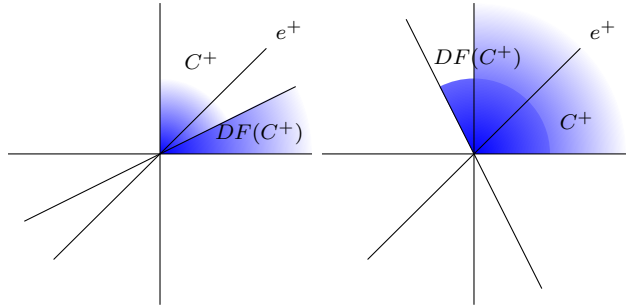
\begin{figure}[h]
\begin{tikzpicture}
\begin{scope}
\clip (0,0) -- (1,0) arc (0:90:1) -- (0,0);
\shade[inner color=blue,outer color=white] (0,0) circle (1);
\end{scope}
\begin{scope}
\clip (0,0) -- (2,0) arc (0:27:2) -- (0,0);
\shade[inner color=blue,outer color=white] (0,0) circle (2);
\end{scope}
\draw (-2,0) -- (2,0);
\draw (0,-2) -- (0,2);
\draw (-1.4,-1.4) -- (1.4,1.4);
\draw (-4/2.24,-2/2.24)-- (4/2.24,2/2.24);
\draw (1.4,1.4) node[anchor=south west] {\footnotesize$e^+$};
\draw (0.2,1) node[anchor=south west] {\footnotesize$C^+$};
\draw (0.6,0) node[anchor=south west] {\footnotesize$DF(C^+)$};
\end{tikzpicture}
\begin{tikzpicture}
\begin{scope}
\clip (0,0) -- (2,0) arc (0:90:2) -- (0,0);
\shade[inner color=blue,outer color=white] (0,0) circle (2);
\end{scope}
\begin{scope}
\clip (0,0) -- (1,0) arc (0:117:1) -- (0,0);
\shade[inner color=blue,outer color=white] (0,0) circle (2.5);
\end{scope}
\draw (-2,0) -- (2,0);
\draw (0,-2) -- (0,2);
\draw (-1.4,-1.4) -- (1.4,1.4);
\draw (-2/2.24,4/2.24)-- (2/2.24,-4/2.24);
\draw (1.4,1.4) node[anchor=south west] {\footnotesize$e^+$};
\draw (1.7,.2) node[anchor=south east] {\footnotesize$C^+$};
\draw (0,1) node[anchor=south] {\footnotesize\ $DF(C^+)$};
\end{tikzpicture}
\caption{Action of a positive and negative twist (shear) on the first quadrant}
\label{fig3}
\end{figure}
Therefore, if\/ $q>0$, the image of the first and third quadrant (ie the trace of\/ $\overline{C^+}$) is a subcone of the first and third quadrant that shares the horizontal axis (see \cref{fig3}). Roughly speaking, this implies that the cone field $C^+$ is preserved by the surgery and one can define a new cone field one the surgered manifold by
\begin{align*}
Q_0^\pm&=\pm\dd w\dd s-c\dd t^2, &\text{ if } t<0,\\
Q_1^\pm&=\pm\left(\dd w\dd s-b(t)f'(w)\dd w^2\right)-c\dd t^2, &\text{ if } t>0,
\end{align*}
where $b\colon\R\to\R^+$ smooth with $b((-\infty,0])=\{1\}$, $b([\eta,\infty))=\{0\}$ and
$b'<0$ on $(-\eta,\eta)$.
Then for $t=0$, $F^*Q_1=Q_0$ and $Q_0^\pm$ and $Q_1^\pm$ induces a pair of Lyapunov--Lorentz metrics on $M_S$. If\/ $q<0$, the cones are not preserved and the flow is not Anosov:

\begin{proposition}\label{PRPNotAnosovIfq<0}
  The \((1,q)\)-Dehn surgery defined by \(F\) in \eqref{eqdefF} does not produce an Anosov flow if \(-q/\epsilon\) is large enough, i.e., if either $q<0$ is fixed and $\epsilon$ is small enough or if $\epsilon>0$ is fixed and $q<0$ with $|q|$ big enough.
\end{proposition}
\begin{proof}
There is a lower bound on the return time to the surgery region, so there is a $K>0$ such that the half-cone \(a\le-Kb\le0\) is mapped into the half-cone \(0\ge a\ge Kb\) by the differential of the return map (see \cref{FIGPRPNotAnosovIfq<0}). Here, we use coordinates \((a,b)\) in the \((s,w)\)-plane. Now suppose that $q/\epsilon<-2K$ and that the function \(g\) in the definition of \(f\) (after \eqref{eqdefF}) is chosen with monotone derivative on \((0,\infty)\). Then \(f'(0)<q/\epsilon\), so \(f'(w)<q/\epsilon\) for small \(w\). This has the effect that for such \(w\), the half-cone around \(e^+\) given  by \(0\le a\le Kb\) is mapped by \(DF\) into the half-cone \(a\le-Kb\le0\), which is on the other side of \(e^-\). The return map then sends it into the half-cone \(0\ge a\ge Kb\), which is the other half of the cone in which we started. This is incompatible with the existence of a \emph{continuous} invariant cone field that extends to points that miss the surgery region, and hence with the Anosov property.
\end{proof}
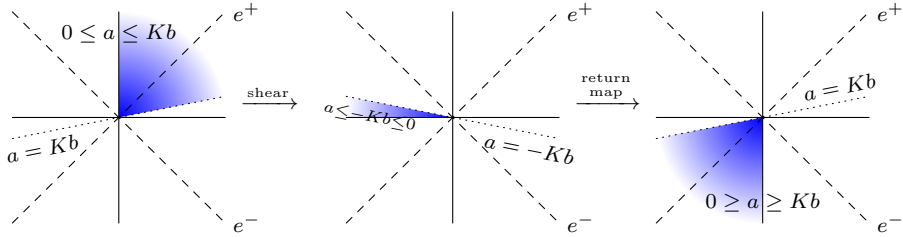
\begin{figure}[ht]\footnotesize
\begin{tikzpicture}[line cap=round,line join=round,>=triangle 45,x=1.4cm,y=1.4cm]
\begin{scope}
\clip (0,0) -- (1,.2) arc (11.46:90:1.02) -- (0,0);
\shade[inner color=blue,outer color=white] (0,0) circle (1);
\end{scope}
\draw (-1,0) -- (1,0) node[anchor=south west] {{$\ \xrightarrow{\text{\Tiny shear}}$}};
\draw (0,-1) -- (0,1) node[anchor=south] {};
\draw[style=dashed] (-1,-1) -- (1,1) node[anchor=west] {$e^+$};
\draw[style=dashed] (-1,1) -- (1,-1) node[anchor=west] {$e^-$};
\draw[style=dotted] (-1,-.2) -- node[very near start,sloped,below]{$a=Kb$} (1,.2);
\node at (0,.8) {$0\le a\le Kb$};
\end{tikzpicture}
\begin{tikzpicture}[line cap=round,line join=round,>=triangle 45,x=1.4cm,y=1.4cm]
\begin{scope}
\clip (0,0) -- (-1,.2) arc (0:-11.46:1.02) -- (0,0);
\shade[inner color=blue,outer color=white] (0,0) circle (1);
\end{scope}
\draw (-1,0) -- (1,0) node[anchor=south west] {$\ \xrightarrow{\substack{\text{\Tiny return}\\\text{\Tiny map}}}$};
\draw (0,-1) -- (0,1) node[anchor=south] {};
\draw[style=dashed] (-1,-1) -- (1,1) node[anchor=west] {$e^+$};
\draw[style=dashed] (-1,1) -- (1,-1) node[anchor=west] {$e^-$};
\draw[style=dotted] (-1,.2) -- node[very near start,sloped,below]{$\scriptstyle a\le-Kb\le0$} node[very near end,sloped,below]{$a=-Kb$} (1,-.2);
\end{tikzpicture}
\begin{tikzpicture}[line cap=round,line join=round,>=triangle 45,x=1.4cm,y=1.4cm]
\begin{scope}
\clip (0,0) -- (-1,-.2) arc (-168.54:-90:1.02) -- (0,0);
\shade[inner color=blue,outer color=white] (0,0) circle (1);
\end{scope}
\draw (-1,0) -- (1,0) node[anchor=south west] {};
\draw (0,-1) -- (0,1) node[anchor=south] {};
\draw[style=dashed] (-1,-1) -- (1,1) node[anchor=west] {$e^+$};
\draw[style=dashed] (-1,1) -- (1,-1) node[anchor=west] {$e^-$};
\draw[style=dotted] (-1,-.2) -- node[very near end,sloped,above]{$a=Kb$} (1,.2);
\node at (0,-.8) {$0\ge a\ge Kb$};
\end{tikzpicture}
\caption{The cones in the proof of \cref{PRPNotAnosovIfq<0}}\label{FIGPRPNotAnosovIfq<0}
\end{figure}

Note that one can perform a positive surgery on an Anosov flow (and therefore obtain another Anosov flow) then undo it by performing a negative surgery and obtain again an Anosov flow. This is compatible with the satement of \cref{PRPNotAnosovIfq<0}, as $q$ and $\epsilon$ are fixed in the negative surgery (and thus \cref{PRPNotAnosovIfq<0} does not apply).

Returning to the case of positive \(q\), we note from the preceding:
\begin{proposition}\label{PROPFoliationsOrientability}
The stable and unstable foliations of\/ $(M_S,\alpha_A)$ as described in \cref{THMMain} are orientable.
\end{proposition}
\begin{proof}
The strong stable foliation is contained in the
positive cone of\/ $Q^-$ and the strong unstable foliation in the positive cone of\/ $Q^+$, so the stable foliation is orientable if and only if the positive cone of\/ $Q^-$ is orientable (an orientation of the positive cone is a choice of a connected component of this cone). The stable and unstable foliations of the unit tangent bundle over a hyperbolic surface are
orientable. Additionally, $Q^-\left(\frac{\partial}{\partial s},\frac{\partial}{\partial s}\right)=0$ and $F^*\frac{\partial}{\partial s}=\frac{\partial}{\partial s}$, so the surgery preserves the orientation of\/ $Q^-$, and $Q^-$ is orientable. It implies that $Q^+$ is orientable.
\end{proof}

\subsection{Impact on entropy}
The nature of the surgery map then implies:
\begin{proposition}\label{PRPsameLyapunov}
If \(q\ge0\), then $\hmu(X_{\mathit{HT}})\ge\hmu(X)$.
\end{proposition}
\begin{proof}
By the Pesin entropy formula it suffices to show that the positive Lyapunov
exponent of \(X_{\mathit{HT}}\) is no less than that of \(X\). Volume-preserving Anosov
3-flows are ergodic \cite[Theorem 20.4.1]{KatokHasselblatt}, so the
positive Lyapunov exponent, being a flow-invariant bounded measurable
function, is a.e.\ a constant. The earlier observation that for the
geodesic flow on a hyperbolic surface the expanding vector is of the form
\(e^te^+\) means that the Lyapunov exponent of (the normalized) Liouville measure is 1.
Therefore, we will show that the positive Lyapunov exponent of\/ $X_{\mathit{HT}}$ is
at least 1. To that end we verify that the differential of its time-1 map
expands unstable vectors by at least a factor of \(e\) with respect to a
suitable norm.

For the geodesic flow the Sasaki metric induces a natural norm, and this norm is what is called an \emph{adapted} or \emph{Lyapunov} norm: for unstable vectors, this norm grows by exactly \(e^t\) under the flow, and on each tangent space it is a product norm. Our argument involves only vectors in unstable cones, so we pass to a norm \(\|\cdot\|_+\) that is (uniformly) equivalent when restricted to such vectors: the norm of the unstable component. Geometrically, this means that at each point we project tangent vectors to \(E^+\) along \(E^-\oplus\mathbb R X$ and take the length of this unstable projection as the norm of the vector. Thus, \(\|Dg^t(v)\|_+=e^t\|v\|_+\) for \(t\ge0\).

The proof of hyperbolicity of \(X_{\mathit{HT}}\) shows that the cone field defined by
the Lyapunov--Lorentz functions is well-defined on the surgered manifold
and invariant under \(X_{\mathit{HT}}\). Thus, this adapted norm for the geodesic flow defines a (bounded, though discontinuous) norm \(\|\cdot\|_+\) on unstable vectors for the flow \(\varphi^t\) defined by\(X_{\mathit{HT}}\). We now show that \(\|D\varphi^1(v)\|_+\ge e\|v\|_+\) for any \(v\) in an unstable cone. This is clear (with equality) when the underlying orbit segment does not meet the surgery annulus because the action is that of the geodesic flow. If there is an encounter with the surgery annulus at time \(t\in(0,1]\), then \(v'\dfn D\varphi^t(v)\) satisfies \(\|v'\|_+=e^t\|v\|_+\), and we will check that \(v''\dfn DF(v')\) satisfies \(\|v''\|_+\ge\|v'\|\), which implies that \(\|D\varphi^1(v)\|_+=\|D\varphi^{1-t}(v'')\|_+=e^{1-t}\|v''\|_+\ge e^{1-t}\|v'\|_+=e^{1-t}e^t\|v\|_+=e\|v\|_+\), as required.

That \(\|DF(v')\|_+\ge\|v'\|\) follows from the same argument as hyperbolicity of \(X_{\mathit{HT}}\) as suggested by \cref{FIGShearGrowsUnstable}, which superimposes the tangent spaces at some \(x\) and \(F(x)\) in the surgery annulus (using the identification from the canonical isometries between these tangent spaces). \(DF\) is a positive shear, and in the \(H\)-\(V\)-frame in the figure the addition of a multiple of the projection of $\frac{\partial}{\partial s}$ (which is close to $H$) by a positive shear results in an increase in the projection to \(E^+\), which is spanned by \(e^+=V+H\).

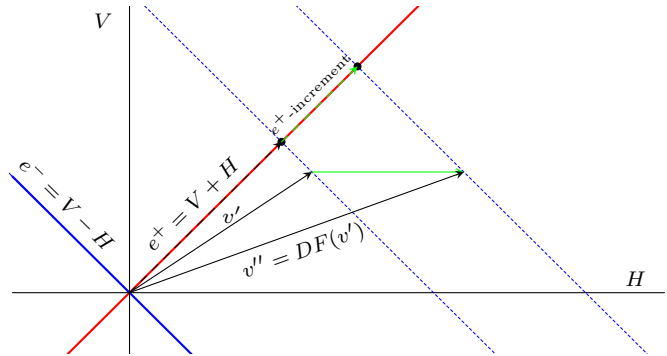
\begin{figure}[ht]
\begin{tikzpicture}[line cap=round,line join=round,>=stealth,x=2cm,y=2cm]
\clip(-0.79,-0.4) rectangle (3.52,1.9);
\draw[red,thick](-0.77,-0.77) -- (3.52,3.52);
\draw [blue] (-1,1) -- (0,0)node[midway,sloped,above,black]{\footnotesize$e^-=V-H$};
\draw[blue,thick](-0.77,0.77) -- (3.52,-3.52);
\draw (0,-0.79) -- (0,3.51);
\draw [domain=-0.77:3.52] plot(\x,{(-0-0*\x)/1});
\draw [dash pattern=on 1pt off 1pt,color=blue,domain=-0.77:3.52]
plot(\x,{(-3--1*\x)/-1});
\draw [dash pattern=on 1pt off 1pt,color=blue,domain=-0.77:3.52]
plot(\x,{(-2--1*\x)/-1});
\fill (1.5,1.5) circle (1.5pt);
\fill (1,1) circle (1.5pt);
\draw [->,dashed] (0,0) -- (1,1)node[midway,sloped,above,black]{\footnotesize$e^+=V+H$};
\draw (-.05,1.8) node[anchor=east] {\footnotesize$V$};
\draw (3.2,.2) node[anchor=north west] {\footnotesize$H$};
\draw [->,color=green] (1.2,.8) -- (2.2,.8)node[midway,above,black]{};
\draw [->] (0,0) -- (1.2,.8)node[midway,above,black]{\footnotesize\quad$v'$};
\draw [->,color=green,dashed] (1,1) -- (1.5,1.5)node[midway,sloped,above,black] {\Tiny $e^+$-increment};
\draw [->] (0,0) -- (2.2,.8)node[midway,sloped,below,black]{\footnotesize$v''=DF(v')$};
\end{tikzpicture}
\caption{\(DF\) increases the unstable component}\label{FIGShearGrowsUnstable}
\end{figure}
\end{proof}

\begin{remark}
Alternatively, let $z$ be a point on the annulus of surgery such that its orbit under the
flow $X_s$ will cross the surgery region infinitely often. As usual consider $z$ as the
identification of $p = (t,s, w)$ with
$F(p) = (t,s +f(w), w)$ if $\xi_0(F(p)) = a e^+ (F(p)) + b e^-(F(p))$ is in the preserved cone given by $ a> 0 , 0 \leq b \leq a$. Consider the first return, at
time $ t$, Then if $ q= \phi_t( F(p)) = (0,s_t, w_t)$, the image
of $\xi_0(F(p))$ by the linear tangent map is
\[ \xi_t(F(q))=  a'e^+( F(q)) + b' e^-(F(q)) +  c' X(F(q))\]
with
\[ a' = a\exp (t) + \frac{df}{dw} (w_t)a_0(a\exp(t)-b\exp(-t)),\]
where
\[\frac{\partial}{\partial s}(F(q)) =  a_0e^+( F(q)) + b_0 e^-(F(q)) +  c_0 X(F(q)) \]
so
$ a' \geq  a( \exp (t) + 2\frac{df}{dw} (w_t)a_0 sht) .$
This gives the desired inequality for the projected norm.
\end{remark}
\begin{remark}
We emphasize that the entropy-increase is manifested for \(X_{\mathit{HT}}\) and thus results from the surgery and not from the time-change that makes the flow contact.
\end{remark}

We are now ready to pursue the growth of periodic orbits.
\begin{proof}[Proof of \cref{THMLargerTopEnt}]
Abramov's formula \eqref{eqAbramov} with \(\displaystyle g\dfn\dfrac{c}{1\pm\dd h(X_{\mathit{HT}})}\) and \(\mu_g\) the normalized volume defined by \(\alpha_A\) gives
\[
\hmu(X_h)=\hmu(X_{\mathit{HT}})\int\dfrac{c}{1\pm\dd h(X_{\mathit{HT}})}\alpha\wedge\dd\alpha=\hmu(X_{\mathit{HT}}).
\]
Combined with our previous result, this gives
\begin{equation}\label{eqLiouvilleEntDefect}
\hmu(\varphi^t)=\hmu(X_h)=
\underbrace{\hmu(X_{\mathit{HT}})\ge\hmu(X)}_{\text{\cref{PRPsameLyapunov}}}=\hmu(g^t).
\end{equation}
This in turn yields a comparison of topological entropies:
\[
\rlap{$\underbrace{\phantom{\htop(g^t)=\hmu(g^t)}}_{\text{constant curvature}}$}\htop(g^t)=\rlap{$\overbrace{\phantom{\hmu(g^t)\le\hmu(\varphi^t)}}^{\eqref{eqLiouvilleEntDefect}}$}\hmu(g^t)\le\underbrace{\hmu(\varphi^t)<\htop(\varphi^t)}_{\text{\cref{THMMain}.\ref{itemBMMnotVolume}}}.\qedhere
\]
\end{proof}
\begin{proof}[Proof of \cref{THMLargerSharpGrowth}]
By \eqref{eqSharpCounting}, increased
cohomological pressure suffices:
\[
\overbrace{\hmu(g^t)\le P(g^t)\le\htop(g^t)}^{\text{\eqref{eqFanginequality}}}=\rlap{$\overbrace{\phantom{\hmu(g^t)\le\hmu(\varphi^t)}}^{\eqref{eqLiouvilleEntDefect}}$}\hmu(g^t)\llap{$\underbrace{\phantom{\htop(g^t)=\hmu(g^t)}}_{\text{constant curvature}}$}\le\underbrace{\hmu(\varphi^t)<P(\varphi^t)}_{\text{\cref{THMFangCohomPressure}}}.
\]
Of course, applying \eqref{eqFanginequality} on the right-hand side
reproves \cref{THMLargerTopEnt}.
\end{proof}

\section{Contact homology and its growth rate}\label{SECContHom}
\emph{Contact homology} is an invariant of the contact structure computed through a Reeb vector field and introduced in the vein of Morse and Floer homology by Eliashberg, Givental and Hofer in 2000~\cite{EGH00}.
The definition of contact homology is subtle and complicated. In this paper, we will consider it as a black box and only use the properties of contact homology described in \cref{THMFPCH}\rlap.\footnote{Plus, for \cref{PROPCH} we also use (and detail in the proof) an elementary and standard application of the computation of contact homology in the Morse--Bott setting.}

Roughly speaking, contact homology is the homology of a complex generated by Reeb periodic orbits of a (nice) contact form. Yet the homology does not depend on the choice of a contact form (but it depends on the underlying contact structure). Therefore a Reeb vector field provides us with information on contact homology and vice-versa. The differential of this complex ``counts'' rigid holomorphic cylinders in the symplectization $M\times \mathbb R $ of our contact manifold (this is the technical part of the definition). These cylinders are asymptotic to Reeb periodic orbits when the $\mathbb R$-coordinate of the cylinder tends to $\pm\infty$. Roughly speaking, if a rigid cylinder is asymptotic to $\gamma_\pm$ at $\pm\infty$, then it contributes $\pm1$ to the coefficient of $\gamma_-$ in the differential of $\gamma_+$. This can be seen as a generalization of the differential of Morse Homology where we ``count'' rigid gradient trajectories asymptotic to critical points of a Morse function.
In particular, this implies that the differential of a periodic orbit only involves periodic orbits in the same free homotopy class and with smaller period. Moreover, the complex is \emph{graded} and the differential decreases the degree by $1$ (here we will only use the parity of this grading). Computing this differential is usually out of reach without a strong control of homotopic periodic Reeb orbits.

Variants of contact homology can be defined by considering periodic orbits in specific free homotopy classes or periodic orbits with period bounded by a given positive real number $T$ (this operation is called a filtration). In the later situation, the limit $T\to\infty$ recovers the original homology. This procees is fundamental to gather information on the growth rate of Reeb periodic orbits.

We recall that, if\/ $\gamma$ is a non\-degenerate $T$-periodic orbit of the Reeb flow $\phi^t$ of\/ $(M,\xi=\ker(\alpha))$ and $p$ is a point on $\gamma$, the orbit $\gamma$ is said to be \emph{even} if the symplectomorphism $\dd\phi^T(p)\colon(\xi_p,\dd \alpha)\to(\xi_p,\dd \alpha)$ has two real positive eigenvalues, and \emph{odd} otherwise.

\begin{theorem}[Fundamental properties of cylindrical contact homology]\label{THMFPCH}
Let $(M,\xi)$ be a closed hypertight contact $3$-manifold, $\alpha_0$ a non\-degenerate contact form on \((M,\xi)\) and $\Lambda$ a set of free homotopy classes of\/ $M$,
\begin{enumerate}
\item Cylindrical contact homology $C\mathbb H^\Lambda_\mathrm{cyl}(\alpha_0)$ is a $\mathbb Q$-vector space. It can be of finite or infinite dimension. It is the homology of a complex generated by $R_{\alpha_0}$-periodic orbits in $\Lambda$.
\item\label{itemContHomDiffOddEven} The differential of an odd (resp. even) orbit contains only even (resp. odd) orbits.
\item If\/ $\alpha$ is another non\-degenerate contact form on $(M,\xi)$, then $C\mathbb H^\Lambda_\mathrm{cyl}(\alpha_0)$ and $C\mathbb H^\Lambda_\mathrm{cyl}(\alpha)$ are isomorphic.
\item There exists a filtered version $C\mathbb H^\Lambda_{\leq T}(\alpha_0)$ (for $T\geq 0$) of contact homology: the associated complex is generated only by periodic orbits in $\Lambda$ with period $\leq T$. Therefore, $C\mathbb H^\Lambda_{\leq T}(\alpha_0)$ is a $\mathbb Q$-vector space of finite dimension and
\[\dim\left(C\mathbb H^\Lambda_{\leq T}(\alpha_0)\right)\leq \sharp\left\{R_{\alpha_0}\text{-periodic orbits in } \Lambda \text{ with period} \leq T \right\}\nfd N_T^\Lambda(\alpha_0)\]
\item $(C\mathbb H^\Lambda_{\leq T}(\alpha))_T$ is a directed system and its direct limit is the cylindrical contact homology. Having a directed system means that for all\/ $T\leq T'$, there exists a morphism $\phi_{T,T'}\colon C\mathbb H^\Lambda_{\leq T}(\alpha_0)\longrightarrow C\mathbb H^\Lambda_{\leq T'}(\alpha_0) $ and
\begin{itemize}
\item $\phi_{T,T}=\Id$
\item if\/ $T_0\leq T_1 \leq T_2$, then $\phi_{T_0,T_2}=\phi_{T_1,T_2}\circ\phi_{T_0,T_1}$.
\end{itemize}
As $\lim C\mathbb H^\Lambda_{\leq T}(\alpha_0)=C\mathbb H^\Lambda(\alpha_0)$, there exist morphisms \[\phi_T\colon C\mathbb H^\Lambda_{\leq T}(\alpha_0)\longrightarrow C\mathbb H^\Lambda(\alpha_0)\] such that the following diagram commutes for $T\leq T'$
\begin{center}
\begin{tikzpicture}[shorten >=1pt,node distance=2.0cm,auto]

\node      (A)                       {$C\mathbb H^\Lambda_{\leq T}(\alpha_0)$};
\node      (B) [below right of=A]  {$C\mathbb H^\Lambda(\alpha_0)$};
\node      (C) [above right  of=B] {$C\mathbb H^\Lambda_{\leq T'}(\alpha_0)$};

\path[->] (A) edge              node        {$\phi_{T,T'}$} (C)
edge            node [swap] {$\phi_{T} $ } (B)
(C) edge              node        {$\phi_{T'} $ } (B);
\end{tikzpicture}
\end{center}
\item Let $\alpha=f\alpha_0$ be another non\-degenerate contact form. Assume $f>0$, and let $B$ be such that $1/B\leq f(m)\leq B$ for all\/ $m\in M$. There exist $C=C(B)$ and morphisms $\psi_T\colon C\mathbb H^\Lambda_{\leq T}(\alpha_0)\longrightarrow C\mathbb H^\Lambda_{\leq C T}(\alpha)$ such that the following diagram commutes
\begin{center}
\begin{tikzpicture}[shorten >=1pt, auto]
\node      (A)                       {$C\mathbb H^\Lambda_{\leq T}(\alpha_0)$};
\node      (B) [right =1.2cm of A]         {$C\mathbb H^\Lambda_{\leq CT'}(\alpha)$};
\node      (C) [below =1.1cm of A]         {$C\mathbb H^\Lambda_{\leq T'}(\alpha_0)$};
\node      (D) [right =1.2cm of C]         {$C\mathbb H^\Lambda_{\leq CT}(\alpha)$};

\path[->] (A) edge              node [swap] {$\phi_{T,T'}(\alpha_0)$} (C)
edge              node        {$\psi_{T} $ } (B)
(C)   edge              node        {$\psi_{T'} $ } (D)
(B) edge              node        {$\phi_{CT,CT'}(\alpha)$} (D);
\end{tikzpicture}
\end{center}
This defines a morphism of directed system.
\end{enumerate}
\end{theorem}

Contact homology was introduced by Eliashberg, Givental and Hofer~\cite{EGH00}. The filtration properties come from~\cite{ColinHonda08}. The description in terms of directed systems takes its inspiration from \cite{McLean2010} and is presented in \cite[Section 4]{Vaugon12}.
Though commonly accepted, existence and invariance of contact homology remain unproven in general. This has been studied by many people using different techniques. This paper uses only proved results and follows Dragnev and Pardon approaches. If\/ $\alpha$ is hypertight and $\Lambda$ contains only primitive free homotopy classes, the properties of contact homology described in \cref{THMFPCH} derive from \cite{Dragnev04} (see~\cite[Section 2.3]{Vaugon12}). In the general case, \cref{THMFPCH} can be derived from \cite{Pardon}. Cylindrical contact homology for hypertight contact forms (and possibly nonprimitive homotopy classes) and the action filtration are described in \cite[Section 1.8]{Pardon}. The case of a not hypertight contact form when there exists an hypertight contact form derives from the contact homology of contractible orbits \cite[Section 1.8]{Pardon} and our invariant corresponds to $CH^\mathrm{\Lambda}_\bullet$. Note that when computed through a hypertight contact form, $CH^\mathrm{contr}_\bullet$ is trivial and $CH^\mathrm{\Lambda}_\bullet$ is the cylindrical contact homology. In the not hypertight case, our invariants can be interpreted geometrically using augmentations. This viewpoint is described in \cite[Section 2.4 and Section 4]{Vaugon12}.

Combining the two commutative diagrams from \cref{THMFHC} and the invariance of contact homology we obtain the following inequality.

\begin{proposition}\label{PROPGR}
Let $\alpha_0$ and $\alpha=f\alpha_0$ be two non\-degenerate contact forms on $(M,\xi)$, where $M$ is a closed, $3$-dimensional manifold and $\xi$ is hypertight. Assume $f>0$, and let $B$ such that $1/B\leq f(m)\leq B$ for all\/ $m\in M$. Then
\[ N_L^\Lambda(\alpha)\geq \mathrm{rank}(\phi_L(\alpha)) \geq \mathrm{rank}(\phi_{L/C(B)}(\alpha_0))\]
for all\/ $L>0$.
\end{proposition}

If\/ $C\mathbb H^\Lambda(\alpha_0)$ is well-understood, one can get an easier estimate.

\begin{corollary}\label{CORGR}
Let $\alpha_0$ and $\alpha=f\alpha_0$ be two non\-degenerate contact forms on $(M,\xi)$ where $M$ is a closed, $3$-dimensional manifold and $\xi$ is hypertight. Assume $f>0$, and let $B$ such that $1/B\leq f(m)\leq B$ for all\/ $m\in M$. If
\[C\mathbb H^\Lambda(\alpha_0)=\bigoplus_{R_{\alpha_0}\text{-Reeb periodic orbit $\gamma$ in }\Lambda}\mathbb Q\gamma\]
then,
$ N_L^\Lambda(\alpha)\geq N_{L/C(B)}^\Lambda(\alpha_0)$
for all\/ $L>0$.
\end{corollary}

In fact, one can derive another invariant of contact structures from these properties of contact homology. Two non\-decreasing functions $f\colon\mathbb R_+\to \mathbb R_+$ and $g\colon\mathbb R_+\to \mathbb R_+$ have the same \emph{growth rate type} if there exists $C>0$ such that
\[f\left(\frac{x}{C}\right)\leq g(x)\leq f(Cx)\]
for all~$x\in\mathbb R_+$ (for instance, a function grows exponentially is it is in the equivalence class of the exponential). The \emph{growth rate type of contact homology} is the growth rate of\/ $T\mapsto \mathrm{rank}(\phi_T)$. Two non\-degenerate contact forms associated to the same contact structure have the same growth rate type (by \cref{PROPGR}) and therefore, the growth rate type of contact homology is an invariant of the contact structure. The growth rate of contact homology was introduced in~\cite{BourgeoisColin05}. It ``describes'' the asymptotic behavior with respect to $T$ of the number of Reeb periodic orbits with period smaller than $T$ that contribute to contact homology. For a more detailed presentation one can refer to~\cite{Vaugon12}.

Colin and Honda's conjecture~\cite[Conjecture 2.10]{ColinHonda08} (see \cref{SIntro}) for the contact structures from \cref{THMMain}, and \cref{THMFHC} for non\-degenerate contact forms follow from
\begin{proposition}\label{PROPEvenOrbits}
Let $(M,\xi)$ be a compact contact 3-manifold and assume there exists a contact form $\alpha_0$ on $(M,\xi)$ whose Reeb flow is Anosov with orientable stable and unstable foliations. Then any $R_{\alpha_0}$-periodic orbit is even and hyperbolic.
\end{proposition}
Indeed, by \cref{PROPFoliationsOrientability}, one can apply \cref{PROPEvenOrbits} to $(M_S,\alpha_A)$. Note that $\alpha_A$ is hypertight as the Reeb flow is Anosov. Therefore, the differential in contact homology is trivial (\cref{THMFPCH}.\ref{itemContHomDiffOddEven}.) and for any set $\Lambda$ of free homotopy classes,
\[C\mathbb H_{\mathrm{cyl}}^\Lambda(\alpha_A)=\bigoplus_{R_{\alpha_A}\text{-Reeb periodic orbit $\gamma$ in }\Lambda}\mathbb Q\gamma.\]
Let $\alpha=f\alpha_A$ be non\-degenerate with $f>0$ and let $B$ be such that $1/B\leq f(m)\leq B$ for all\/ $m\in M$.
Applying \cref{CORGR} for $\Lambda=\{\rho\}$, we get
$ N_L^\Lambda(\alpha)\geq N_{L/C(B)}^\Lambda(\alpha_A)$
for any $L>0$. Using  the Barthelm\'e--Fenley estimates from \cite[Theorem~F]{BarthelmeFenley2} we obtain the desired logarithmic growth. This finishes the proof of \cref{THMFHC} in the nondegenerate case. Additionally, the number of periodic orbits of an Anosov flow in primitive homotopy classes grows exponentially with the period. Applying \cref{CORGR} for $\Lambda$ the set of all primitive free homotopy classes in $M_S$ proves the Colin--Honda conjecture for contact structures from \cref{THMMain} and non\-degenerate contact forms.

\begin{proof}[Proof of \cref{PROPEvenOrbits}]
By definition of stable and unstable foliations, $D\phi^T_{\vert \xi}(p)$ has real eigenvalues $\mu$ and $\frac{1}{\mu}$ and the associated eigenspaces are $E^+$ and $E^-$. As the strong stable foliation is orientable, the eigenvalues are positive. Thus $\gamma$ is even and hyperbolic
\end{proof}

\section{Orbit growth in a free homotopy class for degenerate contact forms}\label{SECOrbitGrowthFHC}

In this Section, we prove \cref{THMFHC} for degenerate contact form (the non\-degenerate case is explained in the previous section). The proof derives from the proof of \cite[Theorem~1]{Alves3}. Yet Alves' goal was to obtain one orbit with bounded period in some free homotopy class and not control the number of orbit in this class, and the following result is not explicit in \cite{Alves3}.

\begin{corollary}\label{CORExistencePO}
Let $(M,\xi)$ be a closed manifold and $\alpha_0$ an Anosov contact form on $(M,\xi)$. Let $\rho$ be a primitive free homotopy class of\/ $M$ such that \[C\mathbb H_\mathrm{cyl}^\rho(\alpha_0)=\bigoplus_{R_{\alpha_0}\text{-periodic orbit $\gamma$ in }\rho}\mathbb Q\gamma\neq\{0\}.\]
Then, for any contact form $\alpha=f_\alpha\alpha_0$ on $(M,\xi)$ and for any $R_{\alpha_0}$-periodic orbit of period $T$, there exists an $R_\alpha$-periodic orbit in $\rho$ of period $T'$ with $eT\leq T'\leq ET$ where $e=\min|f_\alpha|$ and $E=\max|f_\alpha|$.
\end{corollary}

\begin{proof}[Proof of \cref{CORExistencePO}]
Fix $0<\epsilon<e$. Without loss of generality, we may assume $f_\alpha>0$.
We follow Alves' proof of Theorem 1 in \cite{Alves3} and consider $\alpha=f_\alpha\alpha_0$ on $(M,\xi)$ non\-degenerate. For any $R>0$, Alves constructs (Step 1) a symplectic cobordism $\mathbb R\times M_S$ between $(E+\epsilon)\alpha_0$ and $(e-\epsilon)\alpha_0$ which corresponds to the symplectization of\/ $\alpha$ on $[-R,R]\times M_S$, and a map
\[\Psi_R\colon C\mathbb H_\mathrm{cyl}^\rho((E+\epsilon)\alpha_0)\longrightarrow C\mathbb H_\mathrm{cyl}^\rho((e-\epsilon)\alpha_0)\] by counting holomorphic cylinders in the symplectic cobordism. As $C\mathbb H_\mathrm{cyl}^\rho(C\alpha_0)$ is canonically isomorphic to $C\mathbb H_\mathrm{cyl}^\rho(\alpha_0)$ for any $C>0$, $\Psi_R$ induces an endomorphism of\/ $C\mathbb H_\mathrm{cyl}^\rho(\alpha_0)$ and Alves proves this endomorphism is, in fact, the identity.

Let $\gamma$ be a $R_{\alpha_0}$-periodic orbit of period $T$. For any $C>0$, it induces a $R_{C\alpha_0}$ periodic orbit $\gamma_C$ of period $CT$. As \[C\mathbb H_\mathrm{cyl}^\rho(\alpha_0)=\bigoplus_{R_{\alpha_0}\text{-Reeb periodic orbit $\gamma$ in }\rho}\mathbb Q\gamma,\] $\Psi_R(\gamma_{E+\epsilon})=\gamma_{e-\epsilon}$ and therefore, there exists a holomorphic cylinder between $\gamma_{E+\epsilon}$ and $\gamma_{e-\epsilon}$. Now as $R$ tends to infinity (Step 2), SFT compactness (see \cite{Alves3}) shows that our family of cylinders breaks and a $R_\alpha$-periodic orbit $\gamma_\epsilon$ of period $T_\epsilon$ appears in a intermediate level. By construction $(e-\epsilon)T\leq T_\epsilon\leq (E+\epsilon)T$. Now, let $\epsilon$ tend to $0$ and use the Arzel\`a-Ascoli Theorem to obtain a $R_\alpha$ periodic orbit $\gamma'$ with period $T'$ such that $eT\leq T'\leq ET$.

If\/ $\alpha$ is degenerate (Step 4), there exists a sequence $(\alpha_n)_{n\in\mathbb N}$ of non\-degenerate contact forms converging to $\alpha$ and the Arzel\`a-Ascoli Theorem can again be applied to obtain the desired periodic orbit.
\end{proof}

\begin{proof}[Proof of \cref{THMFHC} for degenerate contact forms]
As $M_S$ is hyperbolic, there are $a_1, b_1,a_2,b_2>0$ such that
\[\frac{1}{a_2}\ln(T)-c_2\leq N_T^\rho(\alpha_A)\leq
a_1\ln(T)+c_1\]
for all\/ $T>0$ \cite[Theorem~F]{BarthelmeFenley2}.
Let $(\gamma_n)_{n\in\mathbb N}$ be a sequence of\/ $R_{\alpha_A}$-periodic orbits in $\rho$ of periods $(T_n)_{n\in\mathbb N}$ such that
\begin{itemize}
\item $\gamma_0$ is a $R_{\alpha_A}$-periodic orbit in $\rho$ with minimal period;
\item for all\/ $n\geq 0$, $\gamma_{n+1}$ is a $R_{\alpha_A}$-periodic orbit in $\rho$ with period $T_{n+1}>\frac{E}{e}T_n$ and such that there exists no periodic orbit of smaller period satisfying the same conditions.
\end{itemize}
By \cref{CORExistencePO}, for any $n\geq 0$, there exists a $R_\alpha$-periodic orbit $\gamma'_n$ of period $T'_n$ such that $eT_n\leq T'_n\leq ET_n$. Therefore,
$T'_n\leq ET_n <eT_{n+1}\leq T'_{n+1}$
for all\/ $n\geq 0$ and all the orbits $\gamma'_n$ are distinct. Thus, $N_{T'_n}^\rho(\alpha)\geq n+1$ for all\/ $n\geq 0$.

To control\/ $N_T^\rho(\alpha)$, we now estimate the growth of\/ $(T_n)_{n\in\mathbb N}$. By definition, for all\/ $n\geq 0$,
\[T_{n+1} = \min\left\{T\mid N_T^\rho(\alpha_0) \geq N_{E/eT_n}^\rho(\alpha_0)+1\right\}.\]
Therefore, if $T$ is such that
\[\frac{1}{a_2}\ln(T)-c_2=a_1\ln(E/eT_n)+c_1+1 \]
then $T_{n+1}\leq T$ and
\[T_{n+1}\leq \left(\frac{E}{e}T_n\right)^{a_1a_2}e^{(1+c_1+c_2)a_2}.\]
Therefore, there exist $a_3, c_3>1$ such that $T_{n+1}\leq (c_3 T_n)^{a_3}$ for all\/ $n\geq 0$.
Thus, there exists $c_4> 0$ such that
\[\ln(T_{n+1})\leq c_4 a_3^{n+1} \]
for all\/ $n\geq 0$ and there exists $c_5\in\mathbb R$ such that
\[\ln(\ln(T_{n+1}))\leq \ln(a_3)(n+1)+c_5\]
for all\/ $n\geq 0$. Now, if\/ $eT_{n-1}\leq T'_{n-1}\leq T\leq T'_n\leq ET_n$, then
\[N_T^\rho(\alpha)\geq n \geq \frac{1}{\ln(a_3)}\ln(\ln(T_{n}))-c_5\geq \frac{1}{\ln(a_3)}\ln(\ln(T))-c_6 \]
for some $c_6\in\mathbb R$. This proves \cref{THMFHC}.
\end{proof}

\begin{remark}
If\/ $a_1a_2=1$, one can get better estimates and obtain the same growth as in the non\-degenerate case.
\end{remark}

\section{Exponential growth of periodic orbits after surgery on a simple geodesic}\label{SECOrbitGrowth_simple_geodesic}

We now prove \cref{th_courbe_simple_entropie} using the following result by Alves. To state it, we first define the \emph{exponential homotopical growth} of cylindrical contact homology. Let $(M,\xi)$ be a closed contact manifold and $\alpha_0$ a hypertight contact form on $(M,\xi)$. For $T>0$, let $N_T^{\mathrm{cyl}}(\alpha_0)$ be the number of free homotopy classes $\rho$ of\/ $M$ such that
\begin{itemize}
\item all the $R_{\alpha_0}$-periodic orbits in $\rho$ are simply-covered, non\-degenerate and have period smaller than $T$;
\item $C\mathbb H_\mathrm{cyl}^\rho(\alpha_0)\neq 0$.
\end{itemize}

\begin{definition}[Alves~\cite{Alves1}]
The cylindrical contact homology of\/ $(M,\alpha_0)$ has \emph{exponential homotopical growth} if there exist $T_0\geq 0$, $a>0$ and $b\in\mathbb R$ such that, for all\/ $T\geq T_0$,
\[N_T^{\mathrm{cyl}}(\alpha_0)\geq e^{aT+b}.\]
\end{definition}

\begin{theorem}[Alves~\cite{Alves1}, Theorem 2]\label{th_alves}
Let $\alpha_0$ be a hypertight contact form on a closed contact manifold $(M,\xi)$ and assume that the cylindrical contact homology has exponential homotopical growth. Then every Reeb flow on $(M,\xi)$ has positive topological entropy.
\end{theorem}

If\/ $\rho$ is a free homotopy class containing only one $R_{\alpha_0}$-periodic orbit and if this orbit is simply-covered and non\-degenerate, it is a direct consequence of the definition of contact homology that $C\mathbb H_\mathrm{cyl}^\rho(\alpha_0)=\mathbb Q$. Therefore, to prove \cref{th_courbe_simple_entropie}, it suffices to prove the following propositions.

\begin{proposition}\label{prop_ht}
The contact form $\alpha_A$ is hypertight in $M_S$.
\end{proposition}

\begin{proposition}\label{prop_N'}
Let $(M_S,\alpha_A)$ be a contact manifold obtained after a contact surgery along a simple geodesic.
Let $N'_T(\alpha_A)$ be the number of free homotopy classes $\rho$ such that $\rho$ contains only one $R_{\alpha_A}$-periodic orbit and this orbit is simply-covered, non\-degenerate and of period smaller than $T$. Then, there exist $T_0\geq 0$, $a>0$ and $b\in\mathbb R$ such that, for all\/ $T\geq T_0$, $N'_T(\alpha_A)\geq e^{aT+b}$.
\end{proposition}

Indeed, the exponential growth of\/ $N'_T(\alpha_A)$ with respect to $T$ induces the exponential homotopical growth of\/ $(M_S,\alpha_A)$ and we can apply \cref{th_alves}.

We now turn to the proofs of \cref{prop_ht} and \cref{prop_N'}.
In $S\Sigma$, $\mathbb T = \pi^{-1}(c)$ is a torus, and our surgery preserves this torus. Let $\mathbb T_S$ denote the associated torus in $M_S$. Van Kampen's Theorem tells us that $\mathbb T_S$ is $\pi_1$-injective.

To prove \cref{prop_N'}, we want to find free homotopy classes with only one periodic Reeb orbit. We will consider free homotopy classes containing a periodic orbit disjoint from $\mathbb T_S$ and prove there are enough of such classes. First, we describe Reeb periodic orbits and study the properties of free homotopies between them.

\begin{claim}\label{po_type}
There are three types of\/ $R_{\alpha_A}$-periodic orbits:
\begin{enumerate}
\item periodic orbits contained in $\mathbb T_S$, the only periodic orbits of this kind are $\geodesic$, $-\geodesic$ ($\geodesic$ with the reverse orientation) and their covers,
\item periodic orbits disjoint from $\mathbb T_S$, these orbits correspond to closed geodesics in $\Sigma$ disjoint from $\pi(\geodesic)$ (this includes multiply-covered geodesics),
\item periodic orbits intersecting $\mathbb T_S$ transversely.
\end{enumerate}
\end{claim}

Therefore, a free homotopy between two $R_{\alpha_A}$-periodic orbits can always be perturbed to be transverse to $\mathbb T_S$.

\begin{proposition}\label{prop_H}
Let $\delta_0,\delta_1$ be two smooth loops in $M_S$ and $H\colon  [0,1]\times S^1\to M_S$ be a free homotopy between $\delta_0$ and $\delta_1$ transverse to $\mathbb T_S$. $N\dfn H^{-1}(\mathbb T_S)$ is a smooth manifold of dimension $1$ properly embedded in $[0,1]\times S^1$. Therefore,
\begin{enumerate}
\item\label{p1} one can modify $H$ so that $N$ does not contain contractible circles,
\item if\/ $\delta_0$ is a $R_{\alpha_A}$-periodic orbit transverse to $\mathbb T_S$, $N$ does not contain a segment with end-points on $\{0\}\times S^1$.
\end{enumerate}
\end{proposition}

\begin{proof}
Consider an innermost contractible circle $c_0$ in $N\subset [0,1]\times S^1$, $c_0$ bounds a disk $D_0$ in  $[0,1]\times S^1$. The image of\/ $c_0$ is contractible in $\mathbb T_S$ as $\mathbb T_S$ is $\pi_1$-injective. Therefore, there exists a continuous $G\colon  D_0\to \mathbb T_S$ such that $H_{|c_0}=G_{|c_0}$ and one can replace $H_{|D_0}$ by $G$ to obtain a new homotopy (still denoted by $H$ ) between $\delta_0$ and $\delta_1$. Now, consider a neighborhood $[-\nu,\nu]\times \mathbb T_S$ of\/ $\mathbb T_S$ in $M_S$ (with $\mathbb T_S\simeq \{0\}\times\mathbb T_S$) and a disk $D_1$ containing $D_0$ such that $H(D_1)\subset [0,\nu]\times \mathbb T_S$ and  $H(D_1\setminus D_0)\subset (0,\nu]\times \mathbb T_S$. One can perturb $H$ in $\mathrm{int}(D_1)$ so that $H(D_1)\subset (0,\nu]\times \mathbb T_S$. Performing this inductively on the contractible circles proves~\ref{p1}.

We now assume $\delta_0$ is an $R_{\alpha_A}$-periodic orbit transverse to $\mathbb T_S$. By contradiction, consider an innermost segment $c_0$ in $N$ with end-points on $\{0\}\times S^1$. The end-points of\/ $c_0$ correspond to consecutive intersection points of\/ $\delta_0$ with $\mathbb T_S$. Let $c_1$ be the segment in $\{0\}\times S^1$ joining these two end-points points and homotopic (relative to end-points) to $c_0$. By construction, there exists a homotopy $(\eta_t)_{t\in[0,1]}\colon  [0,1]\to M_S$ (relative to end-points) between $\eta_0=H(c_0)$ et $\eta_1=H(c_1)$ such that $\eta_t(s)\in \mathbb T_S$ if and only if\/ $t=1$ or $s=0,1$. Let $M'$ be the manifold with boundary obtained by cutting $M_S$ along $\mathbb T_S$. Note that $M'$ can also be obtained by cutting $S\Sigma$ along $\mathbb T$. The projection $M'\to M_S$ is injective in the interior of\/ $M'$, therefore $\eta_t(s)$ is well-defined in $M'$ if\/ $t\neq 0$ and $s\neq 0,1$. Thus, there exists a homotopy $\eta'_t$ in $M'$ lifting $\eta_t$. This homotopy induces a homotopy in $S\Sigma$ and, as a result, a homotopy in $\Sigma$ between a geodesic arc contained in $\pi(\geodesic)$ and a geodesic arc with end-points on $\pi(\geodesic)$. As $\Sigma$ is hyperbolic, this can only happen if our second geodesic arc is also contained in $\pi(\geodesic)$, a contradiction.
\end{proof}

\begin{proof}[Proof of \cref{prop_ht}]
By contradiction, assume there exists a free homotopy $H$ between $\delta$, a $R_{\alpha_A}$ periodic orbit, and a point $p\notin \mathbb T_S$. As $\mathbb T_S$ is $\pi_1$-injective, $\delta$ cannot be contained in $\mathbb T_S$. Without loss of generality we may assume that $H$ is transverse to $\mathbb T_S$ and apply \cref{prop_H}.

If\/ $\delta$ is disjoint from $\mathbb T_S$, then $N\subset [0,1]\times S^1$ (see \cref{prop_H}) can only contain circles parallel to the boundary. We will now prove that we can modify $H$ so that $N$ is empty. Let $c_0$ be the circle in $N$ closest to $\{0\}\times S^1$ and let $C$ be the closure of the connected component of\/ $([0,1]\times S^1)\setminus c_0$ containing $\{1\}\times S^1$. Then $H(c_0)$ is an immersed circle contractible in $\mathbb T_S$ and there exists a continuous map $G\colon C\to \mathbb T_S$ such that $G_{|c_0}=H$ and $G_{\{1\}\times S^1}$ is constant. We replace $H_{|C}$ with $G$ to obtain a new homotopy $H$. Now, consider a neighborhood $[-\nu,\nu]\times \mathbb T_S$ of\/ $\mathbb T_S$ in $M_S$ and a neighborhood $C_1$ of\/ $C$ such that $H(C_1)$ is contained in $[0,\nu]\times \mathbb T_S$. We can perturb $H$ so that $H(C_1)$ is contained in $(0,\nu]\times \mathbb T_S$. Therefore we may assume that $N$ is empty and $H$ is an homotopy in $M_S\setminus \mathbb T_S$. It induces an homotopy in $S\Sigma$, a contradiction as the periodic orbits are not contractible in $S\Sigma$.

Finally, we consider the case $\delta$ transverse to $\mathbb T_S$. In this case, $N$ has boundary points on $\{0\}\times \mathbb T_S$ but not on  $\{1\}\times \mathbb T_S$. This contradicts \cref{prop_H}.
\end{proof}

\begin{proposition}
If\/ $\delta$ is a $R_{\alpha_A}$-periodic orbit disjoint from $\mathbb T_S$, then the free homotopy class of\/ $\delta$ contains exactly one $R_{\alpha_A}$-periodic orbit.
\end{proposition}

\begin{proof}
By contradiction, consider a free homotopy $H$ from $\delta$ to $\delta_1$, a distinct $R_{\alpha_A}$-periodic orbit. Without loss of generality, we may assume that $H$ is transverse to $\mathbb T_S$ apply \cref{prop_H}

If\/ $\delta_1$ is disjoint from $\mathbb T_S$, then $N$ can only contain circles parallel to the boundary. If\/ $N$ is empty, $H$ induces a homotopy in $S\Sigma$ and therefore in $\Sigma$. Yet, two closed geodesics on a hyperbolic surface are not homotopic. This proves $N$ is not empty. Let $c_0$ be the circle in $N$ closest to $\{0\}\times S^1$ and $M'$ be the manifold with boundary obtained by cutting $M_S$ along $\mathbb T_S$. The homotopy $H$ induces a homotopy $G$ between $\delta$ and $H(c_0)\subset \mathbb T_S$. The homotopy $G$ lifts to $M'$ and therefore induces a free homotopy in $S\Sigma$ and, as a result, a free homotopy in $\Sigma$ between a closed geodesics and a loop contained in the geodesic $\pi(\geodesic)$. This can happen only if our first geodesic is a cover of\/ $\pi(\geodesic)$. Yet this implies $\delta\subset\mathbb T_S$, a contradiction.

If\/ $\delta_1$ is transverse to $\mathbb T_S$, the manifold $N$ is not empty and has end-points on $\{1\}\times S^1$ but cannot have end-points on $\{0\}\times S^1$. This contradicts \cref{prop_H}.

Finally, the case $\delta_1$ contained in $\mathbb T_S$ is similar to the case $\delta_1$ disjoint from $\mathbb T_S$. In this case, $N$ contains only circles parallel to the boundary and $\{1\}\times S^1$ is in $N$.
\end{proof}

\begin{proof}[Proof of \cref{prop_N'}]
If\/ $\pi(\geodesic)$ is nonseparating, by cutting $\Sigma$ along $\pi(\geodesic)$ we obtain a surface of genus  at least $1$ with two boundary components. Let $\ell_1$ and $\ell_2$ be two loops in $\Sigma\setminus \geodesic$ homotopically independent and with the same base-point. Then, any nontrivial word in $\ell_1$ and $\ell_2$ defines a nontrivial free homotopy class for $\Sigma$ and there exists a closed geodesic on $\Sigma$ representing this class. This $R_{\alpha_A}$-periodic orbit is always nondegenerate. Additionally, we may assume that the orbits associated to $\ell_1$ and $\ell_2$ are simply-covered.
If a word is not the repetition a smaller word, the associated orbit is therefore simply covered.
As $\ell_1$ and $\ell_2$ are independent all these geodesics are disjoint and their number grows exponentially with the period. Finally, these geodesics do not intersect $\geodesic$ as geodesics always minimize the intersection number.

If\/ $\pi(\geodesic)$ is separating, by cutting $\Sigma$ along $\pi(\geodesic)$ we obtain two surfaces of genus  at least $1$ with one boundary components. The proof is similar.
\end{proof}
\section{Coexistence of diverse contact flows---proof of \cref{th_courbe_simple_polynom}}\label{SMoreContactFlows}
\subsection{Dynamical properties of the periodic Reeb flow after surgery}
We now apply the general construction of contact surgery along a Legendrian curve described in \cref{SScontact} to the contact structure with contact form $\beta$ and periodic Reeb flow described in \cref{SIMoreContactFlows}. On the unit tangent bundle of a hyperbolic surface $\Sigma$, select a closed geodesic $\geodesic\colon S^1=\mathbb R/\mathbb Z\to\Sigma$, $s\mapsto\geodesic(s)$ and consider the Legendrian knot $\gamma$ obtained by rotating the unit vector field along $\geodesic$ by the angle $\theta=\pi/2$. Note that the Legendrian knot $\gamma$ is the same as in \cref{SSAnosov} (and is tangent to $H$). To obtain standard coordinates in a neighborhood of\/ $\gamma$ we first consider an annulus $A$ in $S\Sigma$ transverse to the fibers with coordinates $(s,w)\in S^1\times (-2\epsilon,2\epsilon)$ such that $\beta_{|A}=w\dd s$ and then flow along the Reeb vector field $R_\beta$ to obtain coordinates $(t,s,w)\in S^1\times A=N$ such that $\beta=\dd t + w\dd s$\footnote{These coordinates along $\gamma$ are different from the coordinates defined for the surgery associated to the contact form $\alpha$ as, for instance, the surgery annulus is different. It is possible to derive a contact form from $\beta$ on the surgered manifold using the coordinates and surgery associated to $\alpha$: write $\beta$ in local coordinates, compute $F^*\beta$ and interpolate using bump and cut-off functions. Unfortunately, this construction yields a complicated Reeb vector field. Note that the contact structure obtained this way is isotopic to $\ker(\beta_S)$. This can be proved as follows. First the two surgeries result in the same manifold. Moreover, a surgery can be described as the gluing of a solid torus on an excavated manifold. Therefore we just need to prove that the contact structures on the glued tori are the same. This derives from the classification of contact structures on $\mathbb D^3$ by Eliashberg. See \cite{MakarLimanov} for an application to the torus.} (to remain coherent with previous conventions our circles have different lengths, more precisely $t\in \mathbb R/2\pi \mathbb Z$ and $s\in \mathbb R/\mathbb Z$). Note that $N$ can be interpreted as the suspension of the annulus $A$ by the identity map.

Our non-trivial surgery is defined by a twist (shear) $F$ along $A$. We denote by $M_S$ the manifold $S\Sigma$ after surgery and by $N_S\subset M_S$ the manifold (with boundary) $N$ after surgery. Let $\beta_S$ be the contact form on $M_S$ as described in \cref{SScontact}. Note that $\beta$ and $\beta_S$ coincide outside $N$ and $N_S$ respectively. The manifold $N_S$ is the suspension of the annulus $A$ by the shear map $F$. Moreover, the map $p_S\colon N_S\to(-2\epsilon,2\epsilon)$ given by the $w$-coordinate is well-defined and is a trivial torus-bundle. For $w\in(-2\epsilon,2\epsilon)$, the torus $p_S^{-1}(w)$ is foliated by closed Reeb orbits if and only if
\[f\left(w\right)=2\pi \frac{p_w}{q_w}\in2\pi\mathbb Q\]
where $p_w$ and $q_w$ are coprime. In this situation the orbits of\/ $\frac{\partial}{\partial t}$ on $p_S^{-1}(w)$ are periodic of period $q_w$. The Reeb vector field is a renormalization of\/ $\frac{\partial}{\partial t}$ (see (\ref{eqXh})). Finally, let $\mathbb T = S^1\times S^1\times\{0\}$ in $N$ and $\mathbb T_S$ be its image in $M_S$. By van Kampen's theorem, this torus is incompressible. Therefore the contact form $\beta_S$ is hypertight. Note that if $f(w)\in2\pi\mathbb Q$ and  $f(w')\in2\pi\mathbb Q$ but $f(w)\neq f(w')$, the associated periodic orbits are not freely homotopic.

\subsection{Proof of \cref{th_courbe_simple_polynom}}

The contact form $\beta_S$ is degenerate and the renormalization from the surgery makes the direct study a bit harder. So, to estimate the growth rate of its contact homology, we will standardize and perturb our contact form.

For any $w\in(-2\epsilon,2\epsilon)$, the vector fields $\frac{\partial}{\partial t}+\frac{f(w)}{2\pi}\frac{\partial}{\partial s}$ and $\frac{\partial}{\partial s}$ generate circles in the torus $p_S^{-1}(w)$. These circles induce a trivialisation of\/ $N_S$.
Let $(\tau,\sigma, w)$ be the coordinates on $N_S$ associated to this trivialisation.
Without loss of generality, we may assume that the map $f$ defining the twist (shear) $F$ is constant on $(-2\epsilon,-\epsilon)\cup(\epsilon,2\epsilon)$, that $f'(w)\neq 0$ for any $w\in(-\epsilon, \epsilon)$ and that $f$ is invariant under reflection with respect to the point $(0,q/2)$. Therefore, for $w$ in $[-2\epsilon,-\epsilon]$, \[\beta_S=\dd\tau+w\dd\sigma\] and for $w$ in $[\epsilon,2\epsilon]$, \[\beta_S=\left(1+\frac{qw}{2\pi}\right)\dd\tau+w\dd\sigma.\]

\begin{lemma}
There exist smooth maps $h_0, k_0\colon  (-2\epsilon, 2\epsilon)\to\mathbb R$ such that \[\beta_0=h_0(w)\dd\tau+k_0(w)\dd\sigma\] is a contact form such that $\beta_0=\beta_S$ for $w$ close to $\pm2\epsilon$ and $R_{\beta_0}$ and $R_{\beta_S}$ are positively collinear on $N_S$. Therefore, $\beta_0$ and $\beta_S$ are isotopic (through contact forms).
\end{lemma}

\begin{proof}
Let $h_0$ and $k_0$ be the maps defined by $k_0(w)=w$ and \[h_0(w)=1+\int_{-2\epsilon}^w f(u)/2\pi\dd u\] for $w\in (-2\epsilon,2\epsilon)$.
As $\int_{-\epsilon}^{\epsilon} f(u)d u = q\epsilon$,  $\beta_0=\beta_S$ for $w\in(\epsilon,2\epsilon)$. Moreover, the contact condition is
\[1+\int_{-2\epsilon}^w f(u)/2\pi\dd u -w f(w)/2\pi>0\]
and this condition is always satisfied for $\epsilon$ small enough.
Additionally, the Reeb vector field is positively collinear to $(k'_0(w),-h'_0(w))=(1,-f(w)/2\pi)$. Finally, as $R_{\beta_0}$ and $R_{\beta_S}$ are positively collinear, $(u\beta_S+(1-u)\beta_0)_{u\in[0,1]}$ is a contact isotopy.
\end{proof}

The contact form $\beta_0$ is degenerate. To estimate the growth rate of its contact homology, we have to perturb it. Our perturbation draws its inspiration from Morse--Bott techniques. To describe our perturbation, we need to fix some notations. The manifold $S\Sigma\setminus p^{-1}((-\epsilon,\epsilon))$ is a trivial circle bundle. Let $S'$ be a surface (with boundary) transverse to the fibers and intersecting each fiber once: $S'$ provides us with a trivialisation $S'\times S^1$ of\/ $S\Sigma\setminus p^{-1}((-\epsilon,\epsilon))$. The surface $S'$ has two boundary components. Let $\phi\colon  S'\to \mathbb R$ be a Morse function such that $\phi=0$ on the boundary of\/ $S'$ and, if\/ $q>0$ (resp. $q<0$), the connected component of\/ $\partial S'$ corresponding to $w=-\epsilon$ is a maximum (resp. a minimum) and the connected component corresponding to $w=\epsilon$ a minimum (resp. a maximum). For any $w$ such that $f(w)=2\pi p(w)/q(w)\in 2\pi\mathbb Q$, we denote by $P(w)$ the period of the $R_{\beta_0}$-periodic orbits foliating $p_S^{-1}(w)$. Note that there exists $C_P>0$ such that $q(w)/C_P\leq P(w)\leq C_P q(w)$, this implies that the number of torus with $w\in(-\epsilon,\epsilon)$ foliated by Reeb periodic orbits with period smaller than $L$ grows quadratically in $L$.

For a contact form $\alpha$, let $\sigma(\alpha)$ denote the action spectrum: the set of periods of the periodic orbits of\/ $R_\alpha$.

\begin{proposition}\label{prop_perturbation}
Let $T>0$, $T\notin \sigma(\beta_0)$. There exists $\beta'=l\beta_0$ with $l\colon  M_S\to\mathbb R_+$ arbitrarily close to 1 such that
\begin{itemize}
\item $\beta'$ is hypertight and nondegenerate
\item the periodic orbits of\/ $R_{\beta'}$ with period $\leq T$ are exactly:
\begin{enumerate}
\item the fibers associated to the critical points of\/ $\phi$ and their multiple of multiplicity $\leq \left\lfloor \frac{T}{2\pi} \right\rfloor$
\item for all\/ $w\in(-\epsilon,\epsilon)$ such that $P(w)<T$, two orbits in $p_S^{-1}(w)$ and their multiple with multiplicity $\leq \big\lfloor \frac{T}{P(w)} \big\rfloor$
\end{enumerate}
\item if\/ $\delta$ is a $R_{\beta'}$-periodic orbit of period $\leq T$ then all the $R_{\beta'}$-periodic orbit in the free homotopy class of\/ $\delta$ are periodic orbits of period $\leq T$.
\end{itemize}
\end{proposition}

\begin{proposition}\label{PROPCH}
If\/ $\delta$ is a simply-covered $R_{\beta'}$-periodic orbit of period $\leq T$ of the second type in \cref{prop_perturbation}, then \[C\mathbb H_{\text{cyl}}^{[\delta]}(M,\ker(\beta_0))=\mathbb Q^2.\]
\end{proposition}

\begin{proof}[Proof of \cref{prop_perturbation}]
There exists $\nu>0$ such that for any $w\in(-\epsilon,-\epsilon+\nu]\cup[\epsilon-\nu,\epsilon)$, if\/ $f(w)=2\pi p(w)/q(w)\in 2\pi\mathbb Q$ then $q(w)>C_PT$. Let \[N_S'=p_S^{-1}((-\epsilon,-\epsilon+\nu]\cup[\epsilon-\nu,\epsilon)).\] Let $S''$ be a smooth surface in $M_S$ with boundary obtained by adding to $S'$ two annuli in $N_S$, transverse to $R_{\beta_0}$ and projecting to $[-\epsilon,-\epsilon+\nu]\cup[\epsilon-\nu,\epsilon]$. We can therefore endow $S''\setminus S'$ with coordinates $(s',w')$ such that $w'$ lifts $w$. We now perturb $\phi$ and extend it to $S''$ so that $\phi(s',w')=\phi(w')$ on $S''\setminus S'$, $\phi'(w')\neq 0$ for all\/ $w'\in [-\epsilon,-\epsilon+\nu)\cup(\epsilon-\nu,\epsilon]$, $\phi$ is flat (all its derivative are equal to $0$) for $w=\pm(\epsilon-\nu)$ and the critical points of\/ $\phi$ are unaltered. Finally, we extend $\phi$ to $M_S$ to obtain a smooth function, $R_{\beta_0}$-invariant and such that $\phi\equiv 0$ in $N_S\setminus N_S'$.

Let $\beta_\lambda=(1+\lambda \phi)\beta_0$. This is a standard Morse--Bott perturbation (see \cite[Lemma 2.3]{Bourgeois02}) in $M_S\setminus p_S^{-1}((-\epsilon,\epsilon))$, therefore, for $\lambda\ll1$, the periodic orbits in this area correspond to the critical points of\/ $k$.

In the coordinates $(\tau,\sigma,w)$, we have
\[\beta_\lambda=(1+\lambda \phi(w))(h_0(w)\dd \tau+k_0(w)\dd \sigma). \]
Therefore, in these coordinates, the Reeb vector field is positively collinear to
\[\left((1+\lambda \phi(w))k_0'(w)+\lambda \phi'(w)k_0(w)\right)
\frac{\partial}{\partial \tau}-\left((1+\lambda \phi(w))h_0'(w)+\lambda \phi'(w)h_0(w)\right) \frac{\partial}{\partial \sigma}.\]
The $\sigma$-coordinate is nonzero as $\phi$ and $h$ have the same monotonicity.
For $\lambda\ll1$, the $\sigma$-coordinate is close to $-h'_0(w)$, the $\tau$-coordinate to $k'_0(w)$ and $R_{\beta_\lambda}$ is close to $R_{\beta_0}$.
Therefore, for $\lambda\ll1$, if there is a $R_{\beta_\lambda}$-periodic orbit in $N'_S$, this orbit has slope $2\pi p'(w)/q'(w)\in2\pi\mathbb Q$ with $q'(w)>C_PT$.
Thus there are no periodic orbit with period smaller than $T$ in $N'_S$ and the periodic orbits with period bigger than $T$ are not in the free homotopy classes of orbits with period smaller than $T$ as described in \cref{prop_perturbation}.

In $p_S^{-1}([-\epsilon+\nu,\epsilon-\nu])$, the periodic orbits with period $\leq T$ are contained in tori $p_S^{-1}(w)$ such that $P(w)\leq T$. These tori are foliated by periodic orbits. Morse--Bott techniques apply here and give the second type of periodic Reeb orbits: for any such $w$ we perturb $\beta$ in a neighborhood of\/ $p_S^{-1}(w)$ with a function derived from a Morse function $\phi_w$ defined on $p_S^{-1}(w)/\text{Reeb flow} = S^1$ and the periodic orbits after perturbation correspond to the critical points of\/ $\phi_w$. For a given $w$ we obtain two orbits (one associated to the maximum of\/ $\phi_w$ and one associated to the minimum of\/ $\phi_w$), their covers and some orbits with period bigger than $T$ and in the free homotopy class of arbitrarily large covers of our two simple orbits. This perturbation derives from \cite[Lemma 2.3]{Bourgeois02} and is described for tori in \cite[Section 3.1]{Vaugon12}.

Lastly, standard perturbation techniques prove there exists an arbitrarily small perturbation of\/ $\beta_\lambda$ with the following properties:
\begin{itemize}
\item it gives rise to a nondegenerate contact form,
\item it does not change the periodic orbits with period smaller than $T$,
\item it does not create periodic orbits of period bigger than $T$ in the free homotopy classes of orbits of period smaller than $T$.\qedhere
\end{itemize}
\end{proof}
\begin{proof}[Proof of \cref{PROPCH}]
Let $\delta\in p_S^{-1}(w)$ be a $R_{\beta'}$-periodic orbit of period $\leq T$ of the second type in \cref{prop_perturbation}. Then the $R_{\beta_0}$-periodic orbit in the class $[\delta]$ are exactly the orbits in $p_S^{-1}(w)$ (and all these orbits have the same period). As $\delta$ is simply-covered, Dragnev's~\cite{Dragnev04} results can be applied. Additionally, standard perturbations do not create contractible periodic Reeb orbits. Therefore, the differential for contact homology can be described using ``cascades'' from Bourgeois' work~\cite{Bourgeois02}. The case of a unique torus of orbit is explained in \cite[Section 9.4]{Bourgeois02}. The cascades used to describe the differential in this degenerate setting mix holomorphic cylinders between orbits and gradient lines for $\phi_w$ in $p_S^{-1}(w)/\text{Reeb flow} = S^1$ (for some generic metric). As all periodic orbits in this class have the same period, there is no homolorphic cylinder in the cascade and the differential coincides with the Morse--Witten differential for $\phi_w$ (ie  the differential associated to Morse homology). Therefore, cylindrical contact homology in the free homotopy class $\rho$ is $2$-dimensional. The cascades of Morse--Bott homology are explicitly described in~\cite{BourgeoisOancea} (in a slightly different setting).
\end{proof}

\begin{proof}[Proof of \cref{th_courbe_simple_polynom}]
Let $\beta'=f\beta$ be a nondegenerate hypertight contact form and $B$ be such that $1/B<f<B$. Let $(T_i)_{i\in\mathbb N}$ be an increasing sequence such that $\lim_{i\to+\infty}T_i=+\infty$ and $T_i\notin\sigma(\beta_0)$. For all\/ $i\in\mathbb N$, let $\beta_i=f_i\beta$ be the contact form given by \cref{prop_perturbation} for $T=T_i$.
We may assume, \[\frac{1}{B}<\frac{f_i}{f}<B\] as $f_i$ is arbitrarily close to 1. By \cref{prop_perturbation},
\[\dim\left(C\mathbb H_{T_i}(\alpha_i)\right)\leq \left\lfloor \frac{T_i}{2\pi} \right\rfloor C + 2\sum_{w, P(w)\leq T_i} \left\lfloor \frac{T_i}{P(w)} \right\rfloor\]
where $C$ is the number of critical points of\/ $k$ and
\[\sum_{w,  P(w)\leq T} \left\lfloor \frac{T}{ P(w)} \right\rfloor= O(T_i^2). \]

In addition, we have the following commutative diagram (see \cref{THMFPCH}),
\begin{center}
\begin{tikzpicture}[shorten >=1pt,node distance=3cm, auto]
\node      (A)                       {$C\mathbb H_{\leq T_i/C(B)}(\beta')$};
\node      (B) [right =1.1cm of A]         {$C\mathbb H_{\leq T_i}(\beta_i)$};
\node      (C) [below =1.2cm of A]         {$C\mathbb H(\beta')$};
\node      (D) [below =1.2cm of B]         {$C\mathbb H(\beta_i)$};

\path[->] (A) edge              node [swap] {$\phi_{T_i/C(B)}(\beta')$} (C)
edge              node        {} (B)
(C)   edge              node        {} (D)
(B) edge              node        {$\phi_{T_i}(\beta_i)$} (D);
\end{tikzpicture}
\end{center}
thus
\[rk(\phi_{T_i/C(B)}(\beta'))\leq rk(\phi_{T_i}(\beta_i))\leq \dim\left(C\mathbb H_{T_i}(\beta_i)\right) \leq a_1 (T_i^2).\]
A symmetric commutative diagram implies
\[rk(\phi_{T_i/C(B)}(\beta_i))\leq rk(\phi_{T_i}(\beta'))\]
Propositions \ref{prop_perturbation} and \ref{PROPCH} prove that $\phi_{T_i/C(B)}$ is injective on the class of simply-covered periodic orbits of the second type (as defined in~\cref{prop_perturbation}). Therefore $ rk(\phi_{T_i/C(B)}(\beta_i))\geq a_0 T_i^2$ and the growth rate of contact homology is quadratic.
\end{proof}

\end{document}